\RequirePackage{fix-cm} 
\documentclass[a4paper, twoside, reqno, dvips, 12pt]{amsart}
\usepackage{fixltx2e}   

\usepackage{etex}

\usepackage[latin1]{inputenc}
\usepackage[T1]{fontenc}

\usepackage{listings}
\usepackage{algorithm}
\usepackage[noend]{algpseudocode}

\usepackage[titletoc,title]{appendix}

\usepackage{esint}
\usepackage{dsfont}
\usepackage{xspace}
\usepackage{amsgen}
\usepackage{amsthm}
\usepackage{amssymb}
\usepackage{amsmath}
\usepackage{wasysym}
\usepackage{upgreek}
\usepackage{amsfonts}
\usepackage{stmaryrd}
\usepackage{mathtools}

\usepackage{relsize}
\usepackage{textcomp}
\usepackage{textgreek}
\usepackage[mathcal, mathscr]{euscript}

\usepackage{mathrsfs}
\DeclareMathAlphabet{\mathscrbf}{OMS}{mdugm}{b}{n}

\usepackage{a4wide}

\usepackage[dvipsnames, table]{xcolor}
\definecolor{bckg}{RGB}{20.8, 20.8, 20.8}
\definecolor{oneblue}{rgb}{0.0, 0.0, 0.85}
\definecolor{Lightblue}{RGB}{214, 214, 214}
\definecolor{bluepigment}{rgb}{0.2, 0.2, 0.6}
\definecolor{charcoal}{rgb}{0.21, 0.27, 0.31}
\definecolor{denimblue}{rgb}{0.08, 0.38, 0.74}
\definecolor{Lightgray}{rgb}{0.89, 0.89, 0.89}
\definecolor{darkgrey}{rgb}{0.273, 0.281, 0.30}
\definecolor{darkelectricblue}{rgb}{0.33, 0.41, 0.47}

\usepackage{multirow}

\usepackage[sort&compress, comma, square, numbers]{natbib}

\usepackage{psfrag}
\usepackage{graphicx}
\usepackage{adjustbox}
\usepackage[tight]{subfigure}
\usepackage{morefloats}
\usepackage{indentfirst}

\usepackage[usenames, dvipsnames, pdf]{pstricks}
\usepackage{epsfig}
\usepackage{pst-grad} 
\usepackage{pst-plot} 

\usepackage{rotating}
\usepackage{pdflscape}

\usepackage{acronym}
\usepackage{microtype}
\usepackage[labelsep=period,%
            labelfont={bf,sf,color=bluepigment},%
            justification=raggedright]{caption}

\usepackage[perpage, symbol]{footmisc}

\usepackage[colorlinks,
           urlcolor=oneblue,
           linkcolor=denimblue,
           citecolor=NavyBlue,
           bookmarksopen=false,
           pdfpagemode=UseNone,
           pagebackref]{hyperref}

\usepackage[explicit]{titlesec}

\titleformat{\section}[block]
  {\color{NavyBlue}\Large\sffamily\bfseries}
  {}
  {0.0em}
  {\colorbox{bckg!5}{\strut\parbox{\dimexpr\linewidth-2\fboxsep\relax}{\thesection. #1}}}
  [\vspace*{0.33em}]

\titleformat{name=\section,numberless}[block]
  {\color{NavyBlue}\Large\sffamily\bfseries}
  {}
  {0.0em}
  {\colorbox{bckg!5}{\strut\parbox{\dimexpr\linewidth-2\fboxsep\relax}{#1}}}
  [\vspace*{0.33em}]

\titleformat{\subsection}
  {\color{NavyBlue}\large\sffamily\bfseries}
  {}
  {0.0em}
  {\colorbox{bckg!5}{\parbox{\dimexpr\linewidth-2\fboxsep\relax}{\thesubsection. #1}}}
  [\vspace*{0.33em}]

\titleformat{name=\subsection,numberless}
  {\color{NavyBlue}\Large\sffamily\bfseries}
  {}
  {0em}
  {\colorbox{bckg!5}{\parbox{\dimexpr\linewidth-2\fboxsep\relax}{#1}}}
  [\vspace*{0.33em}]

\titleformat{\subsubsection}
  {\color{bluepigment}\sffamily\normalsize\bfseries}
  {\thesubsubsection}
  {0.5em}
  {#1}
  [\vspace*{0.33em}]

\titleformat{\paragraph}[runin]
  {\color{bluepigment}\sffamily\small\bfseries}
  {}
  {0em}
  {#1}

\titlespacing{\section}{0.0em}{1.5em plus 2pt minus 2pt}%
{1.0em plus 2pt minus 2pt}[0em]
\titlespacing{\subsection}{0.5em}{1.5em plus 2pt minus 2pt}%
{1.0em}[0em]
\titlespacing{\subsubsection}{0.5em}{1.5em plus 2pt minus 2pt}%
{1.0em plus 2pt minus 2pt}[0em]

\usepackage{titletoc}

\contentsmargin{0.5em}
\newlength{\tocsep} 
\titlecontents{section}[\tocsep]
  {\addvspace{10pt}\bfseries\sffamily}
  {\contentslabel[\thecontentslabel]{\tocsep}}
  {}
  {\ \titlerule*[0.75pc]{.}\ \thecontentspage}
  []
\titlecontents{subsection}[\tocsep]
  {\addvspace{8pt}\sffamily}
  {\contentslabel[\thecontentslabel]{\tocsep}}
  {}
  {\ \titlerule*[0.5pc]{.}\ \thecontentspage}
  []
\titlecontents*{subsubsection}[\tocsep]
  {\addvspace{2pt}\footnotesize\sffamily}
  {}
  {}
  {\ \titlerule*[0.35pc]{.}\ \thecontentspage}
  [\\*]

\makeatletter
\def\@setauthors{%
  \begingroup
  \def\thanks{\protect\thanks@warning}%
  \trivlist
  \centering\footnotesize \@topsep30\p@\relax
  \advance\@topsep by -\baselineskip
  \item\relax
  \author@andify\authors
  \def\\{\protect\linebreak}%
  \textsc{\normalsize\textcolor{darkelectricblue}{\authors}}%
  \ifx\@empty\contribs
  \else
    ,\penalty-3 \space \@setcontribs
    \@closetoccontribs
  \fi
  \endtrivlist
  \endgroup
}
\def\@settitle{\begin{center}%
  \baselineskip14\p@\relax
    \bfseries
    \textsc{\Large\textcolor{charcoal}{\@title}}
  \end{center}%
}
\makeatother

\usepackage{enumitem}
\setlist[description]{%
  topsep=30pt,               
  itemsep=5pt,               
  font={\bfseries\sffamily\color{NavyBlue}}, 
}

\usepackage{fancyhdr}
\usepackage{lastpage}

\newcommand*\Title{\textcolor{bluepigment}{On the reducibility and the lenticular sets}}
\newcommand*\Authors{\textcolor{bluepigment}{D.~Dutykh \& J.-L.~Verger-Gaugry}}
\newcommand*{\plogo}{\textcolor{gray}{{\texttt{arXiv.org} / \textsc{hal}}}} 

\numberwithin{equation}{section}

\newtheorem{lemma}{Lemma}

\newtheorem{remark}{Remark}

\newtheorem{example}{Example}
\newtheorem{theorem}{Theorem}

\newtheorem{conjecture}{Conjecture}
\newtheorem{proposition}{Proposition}

\newcommand{\N}{\mathds{N}}
\newcommand{\Q}{\mathds{Q}}
\newcommand{\T}{\mathds{T}}
\newcommand{\Z}{\mathds{Z}}
\newcommand{\B}{\mathcal{B}}
\newcommand{\D}{\mathcal{D}}
\newcommand{\G}{\mathcal{G}}
\newcommand{\K}{\mathcal{K}}
\newcommand{\W}{\mathcal{W}}

\renewcommand{\L}{\mathcal{L}}
\renewcommand{\O}{\mathcal{O}}
\renewcommand{\P}{\mathcal{P}}
\renewcommand{\epsilon}{\varepsilon}

\newcommand{\ue}{\mathrm{e}}
\newcommand{\ui}{\mathrm{i}}

\renewcommand{\leq}{\leqslant}
\renewcommand{\geq}{\geqslant}

\DeclareMathOperator{\lo}{Log\,}

\newcommand{\abs}[1]{\left\lvert\, #1\, \right\rvert}
\newcommand{\norm}[1]{\left\lvert\left\lvert\, #1\, \right\rvert\right\rvert}

\newcommand{\cf}{\emph{cf.}\xspace}
\newcommand{\ie}{\emph{i.e.}\xspace}

\newcommand{\etc}{\emph{etc.}\xspace}

\begin{document}

\title[\Title]{On the Reducibility and the Lenticular Sets of Zeroes of Almost Newman Lacunary Polynomials}

\author[D.~Dutykh]{Denys Dutykh}
\address{\textbf{D.~Dutykh:} Univ. Grenoble Alpes, Univ. Savoie Mont Blanc, CNRS, LAMA, 73000 Chamb\'ery, France and LAMA, UMR 5127 CNRS, Universit\'e Savoie Mont Blanc, Campus Scientifique, F-73376 Le Bourget-du-Lac Cedex, France}
\email{Denys.Dutykh@univ-smb.fr}
\urladdr{http://www.denys-dutykh.com/}

\author[J.-L.~Verger-Gaugry]{Jean-Louis Verger-Gaugry$^*$}
\address{\textbf{J.-L.~Verger-Gaugry:} Univ. Grenoble Alpes, Univ. Savoie Mont Blanc, CNRS, LAMA, 73000 Chamb\'ery, France and LAMA, UMR 5127 CNRS, Universit\'e Savoie Mont Blanc, Campus Scientifique, F-73376 Le Bourget-du-Lac Cedex, France}
\email{Jean-Louis.Verger-Gaugry@univ-smb.fr}
\urladdr{https://www.researchgate.net/profile/Jean\_Louis\_Verger-Gaugry/}
\thanks{$^*$ Corresponding author}


\begin{titlepage}
\thispagestyle{empty} 
\noindent
{\Large Denys \textsc{Dutykh}}\\
{\it\textcolor{gray}{LAMA--CNRS, Universit\'e Savoie Mont Blanc, France}}
\\[0.02\textheight]
{\Large Jean-Louis \textsc{Verger-Gaugry}}\\
{\it\textcolor{gray}{LAMA--CNRS, Universit\'e Savoie Mont Blanc, France}}
\\[0.10\textheight]

\vspace*{4em}

\colorbox{Lightblue}{
  \parbox[t]{1.0\textwidth}{
    \centering\huge\sc
    \vspace*{0.7cm}
    
    \textcolor{bluepigment}{On the Reducibility and the Lenticular Sets of Zeroes of Almost Newman Lacunary Polynomials}

    \vspace*{0.7cm}
  }
}

\vfill 

\raggedleft     
{\large \plogo} 
\end{titlepage}


\newpage
\thispagestyle{empty} 
\par\vspace*{\fill}   
\begin{flushright} 
{\textcolor{denimblue}{\textsc{Last modified:}} \today}
\end{flushright}


\newpage
\maketitle
\thispagestyle{empty}


\begin{abstract}

The class $\B$ of lacunary polynomials $f\,(x)\ :=\ -1\ +\ x\ +\ x^{\,n}\ +\
x^{\,m_{\,1}}\ +\ x^{\,m_{\,2}}\ +\ \ldots\ +\ x^{\,m_{\,s}}\,$, where $s\ \geq\ 0\,$, $m_{\,1}\ -\ n\ \geq\ n\ -\ 1\,$, $m_{\,q+1}\ -\ m_{\,q}\ \geq\ n\ -\ 1$ for $1\ \leq\ q\ <\ s\,$, $n\ \geq\ 3$ is studied. A polynomial having its coefficients in $\{\,0,\,1\,\}$ except its constant coefficient equal to $-1$ is called an \emph{almost} \textsc{Newman} polynomial. A general theorem of factorization of the almost \textsc{Newman} polynomials of the class $\B$ is obtained. Such polynomials possess lenticular roots in the open unit disk off the unit circle in the small angular sector $-\pi/18\ \leq\ \arg\,z\ \leq\ \pi/18$ and their nonreciprocal parts are always irreducible. The existence of lenticuli of roots is a peculiarity of the class $\B\,$. By comparison with the \textsc{Odlyzko--Poonen} Conjecture and its variant Conjecture, an \emph{Asymptotic Reducibility Conjecture} is formulated aiming at establishing the proportion of irreducible polynomials in this class. This proportion is conjectured to be $3/4$ and estimated using Monte-Carlo methods. The numerical approximate value $\approx\ 0.756$ is obtained. The results extend those on trinomials (\textsc{Selmer}) and quadrinomials (\textsc{Ljunggren, Mills, Finch} and \textsc{Jones}).

\bigskip\bigskip
\noindent \textbf{\keywordsname:} Lacunary polynomial; Newman polynomial; Almost Newman polynomial; Reducibility; Lenticular root \\


\end{abstract}


\newpage
\tableofcontents
\thispagestyle{empty}


\newpage
\section{Introduction}

In this note, for $n\ \geq\ 3\,$, we study the factorization of the polynomials
\begin{equation}\label{basic}
  f\,(x)\ :=\ -1\ +\ x\ +\ x^{\,n}\ +\ x^{\,m_{\,1}}\ +\ x^{\,m_{\,2}}\ +\ \ldots\ +\ x^{\,m_{\,s}}\,,
\end{equation}
where $s\ \geq\ 0\,$, $m_{\,1}\ -\ n\ \geq\ n\ -\ 1\,$, $m_{\,q\,+\,1}\ -\ m_{\,q}\ \geq\ n\ -\ 1$ for $1\ \leq\ q\ <\ s\,$. Denote by $\B$ the class of such polynomials, and by $\B_{\,n}$ those whose third monomial is exactly $x^{\,n}\,$, so that
\begin{equation*}
  \B\ =\ \bigcup_{n\,\geq\,2} \B_{\,n}\,.
\end{equation*}
The case $s\ =\ 0$ corresponds to the trinomials $G_{\,n}\,(z)\ :=\ -1\ +\ z\ +\ z^{\,n}$ studied by \textsc{Selmer} \cite{Selmer1956}. Let $\theta_{\,n}$ be the unique root of the trinomial $G_{\,n}\,(z)\ =\ 0$ in $(0,\,1)\,$. The algebraic integers $\theta_{\,n}^{\,-1}\ >\ 1$ are \textsc{Perron} numbers. The sequence $(\theta_{\,n}^{\,-1})_{\,n\ \geq\ 2}$ tends to $1$ if $n$ tends to $+\,\infty\,$.

\begin{theorem}[Selmer \cite{Selmer1956}]
Let $n\ \geq\ 2\,$. The trinomials $G_{\,n}\,(x)$ are irreducible if $n\ \not\equiv\ 5 ~({\rm mod}~ 6)\,$, and, for $n\ \equiv\ 5 ~({\rm mod}~ 6)\,$, are reducible as product of two irreducible factors whose one is the cyclotomic factor $x^{\,2}\ -\ x\ +\ 1\,$, the other factor $(-1\ +\ x\ +\ x^{\,n})\,/\,(x^{\,2}\ -\ x\ +\ 1)$ being nonreciprocal of degree $n\ -\ 2\,$.
\end{theorem}

\begin{theorem}[Verger-Gaugry \cite{Verger-Gaugry2016}]\label{thmthetan}
Let $n\ \geq\ 2\,$. The real root $\theta_{\,n}\ =\ {\rm D}\,(\theta_{\,n})\ +\ {\rm tl}\;(\theta_{\,n})\ \in\ (0,\,1)$ of the trinomial $G_{\,n}$ admits the following asymptotic expansion:
\begin{multline}\label{DbetaAsymptoticExpression}
  {\rm D}\,(\theta_{\,n})\ =\ 1\ -\ \frac{\lo n}{n}\times\\
  \times\left(1\ -\ \Bigl(\frac{n - \lo n}{n\,\lo n + n - \lo n}\Bigr)\Bigl(\lo \lo n - n\lo \Bigl(1 - \frac{\lo n}{n}\Bigr) - \lo n\Bigr)\right)\,,
\end{multline}
and
\begin{equation}\label{tailbetaAsymptoticExpression}
  {\rm tl}\,(\theta_n)\ =\ \frac{1}{n} \cdot \O\,\Bigl(\Bigl(\frac{\lo \lo n}{\lo n}\Bigr)^{\,2}\Bigr)\,,
\end{equation}
with the constant $\frac{1}{2}$ involved in $\O\,\left(\cdot\right)\,$.
\end{theorem}

\begin{remark}
A simplified form of expression \eqref{DbetaAsymptoticExpression} is the following:
\begin{equation}\label{eq:remark1}
  {\rm D}\,(\theta_{\,n})\ =\ 1\ -\ \frac{1}{n}\;\Bigl(\lo n\ -\ \lo \lo n\ +\ \frac{\lo \lo n}{\lo n}\Bigr)\,.
\end{equation}
\end{remark}

By definition a {\em \textsc{Newman} polynomial} is an integer polynomial having all its coefficients in $\{\,0,\, 1\,\}\,$. A polynomial having its coefficients in $\{\,0,\, 1\,\}$ except its constant coefficient equal to $-1$ is called an {\em almost \textsc{Newman} polynomial}. It is not difficult to see that polynomials $f\ \in\ \B$ are almost \textsc{Newman} polynomials. The following irreducibility Conjecture (called \textsc{Odlyzko--Poonen} (OP)) holds for the asymptotics of the factorization of \textsc{Newman} polynomials.

\begin{conjecture}[\textsc{Odlyzko--Poonen} \cite{Odlyzko1993}]
Let $\P_{\,d,\,+}\ :=\ \{1\ +\ \sum_{\,j\,=\,1}^{\,d} a_{\,j}\,x^{\,j}\ \Bigl\vert\ a_{\,j}\ \in\ \{0,\, 1\}\,,\ a_{\,d}\ =\ 1\}$ denote the set of all \textsc{Newman} polynomials of degree $d\,$. We introduce also the class $\P_{\,+}\ :=\ \bigcup_{\,d\ \geq\ 1} \P_{\,d,\,+}\,$. Then, in $\P_{\,+}\,$, almost all polynomials are irreducible. More precisely, if $\Omega_{\,d}$ denotes the number of irreducible polynomials in
$\P_{\,d,\,+}\,$, then
\begin{equation*}
  \lim_{d\ \to\ \infty}\ \frac{\Omega_{\,d}}{2^{\,d\,-\,1}}\ =\ \lim_{d\ \to\ \infty}\ \frac{\#\{f\ \in\ \P_{\,d,\,+}\ \Bigl\vert\ f\ \mbox{is irreducible}\,\}}{2^{\,d\,-\,1}}\ =\ 1\,.
\end{equation*}
\end{conjecture}

The best account of the Conjecture is given by \textsc{Konyagin} \cite{Konyagin1999}:
\begin{equation*}
  \Omega_{\,d}\ \gg\ \frac{2^{\,d}}{\lo d}\,.
\end{equation*}
Replacing the constant coefficients $1$ by $-1$ gives the variant Conjecture (called ``\emph{variant OP}'') for almost \textsc{Newman} polynomials.

\begin{conjecture}(Variant OP)\label{odlyzkopoonenCJ}
Let $\P_{\,d,\,-}\ :=\ \{-1\ +\ \sum_{\,j\,=\,1}^{\,d}\,a_{\,j}\,x^{\,j}\ \Bigl\vert\ a_{\,j}\ =\ 0\ \mbox{or}\ 1\,,\newline a_{\,d}\ =\ 1\}$ denote the set of all almost \textsc{Newman} polynomials of degree $d\,$. Denote
\begin{equation*}
  \P_{\,-}\ =\ \bigcup_{\,d\ \geq\ 1} \P_{\,d,\,-}\,.
\end{equation*}
Then, in $\P_{\,-}\,$, almost all polynomials are irreducible. More precisely, 
\begin{equation*}
  \lim_{d\ \to\ \infty} \frac{\#\{f\ \in\ \P_{\,d,\,-}\ \Bigl\vert\ f \mbox{ is irreducible}\}}{2^{\,d\,-\,1}}\ =\ 1\,.
\end{equation*}
\end{conjecture}
There is a numerical evidence that the OP Conjecture and the variant OP Conjecture are true (\cf Table~\ref{tableIRRED}, Sect.~\textsection\ref{S6}).

The objectives of this note consist in
\begin{enumerate}
  \item Establishing the type of factorization of the polynomials $f$ of the class $\B$ (Theorem~\ref{thm1factorization}), in the context of \textsc{Schinzel}'\,s and \textsc{Filaseta}'\,s theorems on the factorization of lacunary polynomials \cite{Schinzel1969, Schinzel1978, Filaseta1999}.
  \item Characterizing the geometry of the zeroes of the polynomials $f$ of the class $\B\,$, in particular in proving the existence of lenticuli of zeroes in the angular sector $-\pi/18\ \leq\ \arg\,z\ \leq\ \pi/18$ inside the open unit disk in \textsc{Solomyak}'\,s fractal (with numerical examples to illustrate Theorem~\ref{thm2lenticuli}).
  \item Estimating the probability for a polynomial $f\ \in\ \B$ to be irreducible (Heuristics called ``\emph{Asymptotic Reducibility Conjecture}'') by comparison with the variant OP Conjecture.
\end{enumerate}

\bigskip
\paragraph{Notations used in the sequel.} If $P\,(x)\ =\ \sum_{\,j\,=\,0}^{\,r}\, a_{\,j}\,x^{\,r}\ \in\ \Z\,[x]\,$, we refer to the reciprocal polynomial of $P\,(x)$ as $P^{\,*}\,(x)\ =\ \sum_{\,j\,=\,0}^{\,r}\,a_{\,r\,-\,j}\,x^{\,r}\,$. The \textsc{Euclidean} norm $\norm{P}$ of $P\,(x)\ =\ \sum_{\,j\,=\,0}^{\,r}\,a_{\,j}\, x^{\,r}\ \in\ \Z\,[x]$ is $\norm{P}\ :=\ \Bigl(\,\sum_{\,j\,=\,0}^{\,r}\,a_{\,j}^{\,2}\,\Bigr)^{\,\frac{1}{2}}\,$. If $\alpha$ is an algebraic number, $P_{\,\alpha}\,(x)$ denotes its minimal polynomial; if $P_{\,\alpha}\,(x)$ is reciprocal we say that $\alpha$ is reciprocal. A \textsc{Perron} number $\alpha$ is either $1$ or a real algebraic integer $>\ 1$ such that its conjugates $\alpha^{\,(i)}$ are strictly less than $\alpha$ in modulus. The integer $n$ is called the \emph{dynamical degree} of the real algebraic integer $\beta\ >\ 1$ if $1/\beta$ denotes the unique real zero of $f\,(x)\ =\ -\,1\ +\ x +\ x^{\,n}\ +\ \sum_{\,q\,=\,1}^{\,s}\,x^{\,m_{\,q}}\ \in\ \B\,$. Let $\T$ denote the unit circle in the complex plane.

\begin{theorem}\label{thm1factorization}
For any $f\ \in\ \B_{\,n}\,$, $n\ \geq\ 3\,$, denote by
\begin{equation*}
  f\,(x)\ =\ A\,(x)\cdot B\,(x)\cdot C\,(x)\ =\ -1\ +\ x\ +\ x^{\,n}\ +\ x^{\,m_{\,1}}\ +\ x^{\,m_{\,2}}\ +\ \ldots\ +\ x^{\,m_{\,s}}\,,
\end{equation*}
where $s\ \geq\ 1\,$, $m_{\,1}\ -\ n\ \geq\ n\ -\ 1\,$, $m_{\,j\,+\,1}\ -\ m_{\,j}\ \geq\ n\ -\ 1$ for $1\ \leq\ j\ <\ s\,$, the factorization of $f$ where $A$ is the cyclotomic part, $B$ the reciprocal noncyclotomic part, $C$ the nonreciprocal part. Then, 
\begin{enumerate}
  \item the nonreciprocal part $C$ is nontrivial, irreducible and never vanishes on the unit circle,
  \item if $\beta\ >\ 1$ denotes the real algebraic integer uniquely determined by the sequence $(n,\, m_{\,1},\, m_{\,2},\,\ldots,\,m_{\,s})$ such that $1/\beta$ is the unique real root of $f$ in $(\theta_{\,n\,-\,1},\,\theta_{\,n})\,$, the nonreciprocal polynomial $-\,C^{\,*}\,(x)$ of $C\,(x)$ is the minimal polynomial of $\beta\,$, and $\beta$ is a nonreciprocal algebraic integer.
\end{enumerate}
\end{theorem}

\begin{remark}
For all polynomials $f\,$, as described in Theorem~\ref{thm1factorization}, we observe numerically the following lower bound on the degree of the nonreciprocal part $C\,$:
\begin{equation*}
    \deg\,(C)\ \geq\ \lfloor \frac{m_{\,s}\ -\ 1}{2}\rfloor\,,
\end{equation*}
At the current stage this minoration is a conjecture.
\end{remark}

Let us now define the lenticular roots of an $f$ of the class $\B\,$. In the case $s\ =\ 0\,$, \ie for the trinomials $G_{\,n}\,(x)\,$, from \cite[Proposition~3.7]{Verger-Gaugry2008}, the roots of modulus $<\ 1$ of $G_{\,n}$ all lie in the angular sector $-\,\pi/3\ <\ \arg\,z\ <\ \pi/3\,$. The set of these ``internal'' roots has the form of a {\em lenticulus}, justifying the terminology (Figure~\ref{example649}(\textit{a}) for $n\ =\ 37$); they are called {\em lenticular roots}. For extending the notion of ``lenticulus of roots'' to general polynomials $f$ of the class $\B\,$, with $s\ \geq\ 1\,$, we view
\begin{multline*}
  f\,(x)\ =\ -1\ +\ x\ +\ x^{\,n}\ +\ x^{\,m_{\,1}}\ +\ x^{\,m_{\,2}}\ +\ \ldots\ +\ x^{\,m_{\,s}}\ =\\ 
  G_{\,n}\,(x)\ +\ x^{\,m_{\,1}}\ +\ x^{\,m_{\,2}}\ +\ \ldots\ +\ x^{\,m_{\,s}}\,,
\end{multline*}
(where $n\ \geq\ 3\,$, $s\ \geq\ 1\,$, $m_{\,1}\ -\ n\ \geq\ n\ -\ 1\,$, $m_{\,j\,+\,1}\ -\ m_{\,j}\ \geq\ n\ -\ 1$ for $1\ \leq\ j\ <\ s$) as a perturbation of $G_{\,n}\,(x)$ by $x^{\,m_{\,1}}\ +\ x^{\,m_{\,2}}\ +\ \ldots\ +\ x^{\,m_{\,s}}\,$. The lenticulus of roots of $f$ is then a deformation of the lenticulus of roots of $G_{\,n}$ (Figure~\ref{example649}(\textit{b})). In this deformation process, the aisles of the lenticulus may present important displacements, in particular towards the unit circle, whereas the central part remains approximately identical. Therefore it is hopeless to define the lenticulus of roots of $f$ in the full angular sector $-\,\pi/3\ <\ \arg\,\omega\ <\ \pi/3\,$. From the structure of the asymptotic expansions of the roots of $G_{\,n}\,(x)$ \cite{Verger-Gaugry2016} it is natural to restrict to the angular sector to $-\,\pi/18\ <\ \arg\,\omega\ <\ \pi/18\,$. More precisely,
\begin{theorem}[\cite{Verger-Gaugry2018}]\label{thm2lenticuli}
Let $n\ \geq\ 260\,$. There exist two positive constants $c_{\,n}$ and $c_{\,A,\,n}\,$, $c_{\,A,\,n}\ <\ c_{\,n}\,$, such that the roots of $f\ \in\ \B_{\,n}\,$,
\begin{equation*}
  f\,(x)\ =\ -1\ +\ x\ +\ x^{\,n}\ +\ x^{\,m_{\,1}}\ +\ x^{\,m_{\,2}}\ +\ \ldots\ +\ x^{\,m_{\,s}}\,,
\end{equation*}
where $s\ \geq\ 1\,$, $m_{\,1}\ -\ n\ \geq\ n\ -\ 1\,$, $m_{\,j\,+\,1}\ -\ m_{\,j}\ \geq\ n\ -\ 1$ for $1\ \leq\ j\ <\ s\,$, lying in $-\,\pi/18\ <\ \arg\,z\ <\ \pi/18$ either belong to
\begin{equation*}
  \Bigl\{z\ \Bigl\vert\ \abs{\abs{z}\ -\ 1}\ <\ \frac{c_{\,A,\,n}}{n}\Bigr\}\,, \quad \mbox{ or to } \quad \Bigl\{z\ \Bigl\vert\ \abs{\abs{z}\ -\ 1}\ \geq\ \frac{c_{\,n}}{n}\Bigr\}\,.
\end{equation*}
\end{theorem}

The {\em lenticulus of zeroes of $f$} is then defined as
\begin{equation*}
  \L_{\,\beta}\ :=\ \Bigl\{\,\omega\ \Bigl\vert\ \abs{\omega}\ <\ 1\,,\ -\frac{\pi}{18}\ <\ \arg\,\omega\ <\ \frac{\pi}{18}\,,\ \abs{\abs{\omega}\ -\ 1}\ \geq\ \frac{c_{\,n}}{n}\,\Bigr\}\,,
\end{equation*}
where $1/\beta$ is the positive real zero of $f\,$. The proof of Theorem~\ref{thm2lenticuli} requires the structure of the asymptotic expansions of the roots 
of $G_{\,n}$ and is given in \cite{Verger-Gaugry2018}.

A typical example of lenticularity of roots with $n\ =\ 481$ is given in Figure~\ref{ltlent481}, in which
\begin{equation*}
  f\,(x)\ =\ -1\ +\ x\ +\ x^{\,481}\ +\ x^{\,985}\ +\ x^{\,1502}\,.
\end{equation*}

Let $\kappa\ =\ 0.171573\ldots$ be the maximum of the function
\begin{equation*}
  y\ \longmapsto\ \frac{1\ -\ \exp\Bigl(-\;\dfrac{\pi}{y}\Bigr)}{2\,\exp\Bigl(\dfrac{\pi}{y}\Bigr)\ -\ 1}
\end{equation*}
on $(0,\,+\,\infty)\,$. The following formulation for $c_{\,n}$ is given in \cite{Verger-Gaugry2018}:
\begin{equation*}
  c_{\,n}\ =\ -\,\Bigl(\,1\ +\ \frac{1}{n}\,\Bigr)\cdot\lo\kappa\ +\ \frac{1}{n}\cdot\O\,\biggl(\,\Bigl(\frac{\lo \lo n}{\lo n}\Bigr)^{\,2}\,\biggr)\,,
\end{equation*}
with $c_{\,n}\ \simeq\ -\lo\kappa\ \approx\ 1.76274\ldots$ to the first-order. In the present note Theorem~\ref{thm2lenticuli} is only examplified. Namely, in Section~\ref{S4} we show that the statement of this Theorem also holds on examples, in particular pentanomials, for dynamical degrees $n$ less than $260\,$.

Concerning the asymptotic probability of irreducibility of the polynomials of the class $\B$ at large degrees, our numerical results shown in Figure~\ref{IrreducibilityMontecarlo}, using the \textsc{Monte-Carlo} method (see the pseudo-code \ref{algo:mc}), suggest the following
\begin{conjecture}[Asymptotic Reducibility Conjecture]
Let $n\ \geq\ 2$ and $N\ \geq\ n\,$. Let $\B_{\,n}^{\,(N)}$ denote the set of the polynomials $f\ \in\ \B_{\,n}$ such that $\deg\,(f)\ \leq\ N\,$. Let $\B^{\,(N)}\ :=\ \bigcup_{2\ \leq\ n\ \leq\ N}\,\B_{\,n}^{\,(N)}\,$. The proportion of
polynomials in $\B\ =\ \bigcup_{\,N\ \geq\ 2}\,\B^{\,(N)}$ which are irreducible is given by the limit, assumed to exist,
\begin{equation*}
  \lim_{N\ \to\ \infty}\,\frac{\#\{\, f\ \in\ \B^{\,(N)}\ \Bigl\vert\ f\ \mbox{irreducible}\, \}}{\#\{\, f\ \in\ \B^{\,(N)}\,\}}\ =\ \frac{3}{4}\,.
\end{equation*}
\end{conjecture}


\section{Quadrinomials ($s\ =\ 1$)}
\label{S2}

Since every $f\ \in\ \B$ is nonreciprocal and such that $f\,(1)\ \neq\ 0\,$, $f$ is never divisible by the cyclotomic nonreciprocal polynomial $-\,1\ +\ x\,$. When $f\ \in\ \B$ is a quadrinomial, the following Theorems provide all the possible factorizations of $f\,$.

\begin{theorem}[Ljunggren \cite{Ljunggren1960}]
\label{ljunggrenthm1}
If $f\ \in\ \B\,$, as
\begin{equation*}
  f\,(x)\ =\ -\,1\ +\ x\ +\ x^{\,n}\ +\ x^{\,m_{\,1}}\,,
\end{equation*}
has no zeroes which are roots of unity, then $f\,(x)$ is irreducible. If $f\,(x)$ has exactly $q$ such zeroes, then $f\,(x)$ can be decomposed into two rational factors, one of which is cyclotomic of degree $q$ with all these roots of unity as zeroes, while the other is irreducible (and nonreciprocal).
\end{theorem}

\textsc{Ljunggren}'\,s Theorem~\ref{ljunggrenthm1} is not completely correct. \textsc{Mills} corrected it (Theorem~\ref{thmmills}). \textsc{Finch} and \textsc{Jones} completed the results (Theorem~\ref{thmfinchjones}).

\begin{theorem}[Ljunggren \cite{Ljunggren1960}]
If $f\ \in\ \mathcal{B}\,$, with $s\ =\ 1\,$, as 
\begin{equation*}
  f\,(x)\ =\ -\,1\ +\ x\ +\ x^{\,n}\ +\ x^{\,m_{\,1}}\,,
\end{equation*}
with $e_{\,1}\ =\ \gcd\,(m_{\,1},\, n\,-\,1)\,$, $e_{\,2}\ =\ \gcd\,(n,\,m_{\,1}\ -\ 1)\,$, then all possible roots of unity of $f\,(x)$ are simple zeroes, which are to be found among the zeroes of
\begin{equation*}
  x^{\,e_{\,1}}\ =\ \pm\,1\,, \qquad x^{\,e_{\,2}}\ =\ \pm\,1\,, \qquad x\ =\ -\,1\,.
\end{equation*}
\end{theorem}

\begin{theorem}[Ljunggren \cite{Ljunggren1960}]
If $f\ \in\ \B\,$, with $s\ =\ 1\,$, as
\begin{equation*}
  f\,(x)\ =\ -\,1\ +\ x\ +\ x^{\,n}\ +\ x^{\,m_{\,1}}\,,
\end{equation*}
is such that both $n$ and $m_{\,1}$ are odd integers, then $f\,(x)$ is irreducible. 
\end{theorem}

\begin{theorem}[Mills \cite{Mills1985}]\label{thmmills}
Let $f\ \in\ \B\,$, with $s\ =\ 1\,$,
\begin{equation*}
  f\,(x)\ =\ -\,1\ +\ x\ +\ x^{\,n}\ +\ x^{\,m_{\,1}}
\end{equation*}
decomposed as $f\,(x)\ =\ A\,(x)\cdot B\,(x)$ where every root of $A\,(x)$ and no root of $B\,(x)$ is a root of unity. Then $A\,(x)$ is the greatest common divisor of $f\,(x)$ and $f^{\,*}\,(x)\ =\ x^{\,m_{\,1}}\cdot f\,(1/x)\,$, then reciprocal cyclotomic, and the second factor $B\,(x)$ is irreducible, then nonreciprocal, except when $f\,(x)$ has the following form:
\begin{equation*}
  -\,1\ +\ x^{\,r}\ +\ x^{\,7\,r}\ +\ x^{\,8\,r}\ =\ (x^{\,2\,r}\ +\ 1)\cdot(x^{\,3\,r}\ +\ x^{\,2\,r}\ -\ 1)\cdot(x^{\,3\,r}\ -\ x^{\,r}\ +\ 1)\,.
\end{equation*}
In the last case, the factors $x^{\,3\,r}\ +\ x^{\,2\,r}\ -\ 1$ and $x^{\,3\,r}\ -\ x^{\,r}\ +\ 1$ are (nonreciprocal) irreducible.
\end{theorem}

\begin{theorem}[Finch -- Jones, \cite{Finch2006}]\label{thmfinchjones}
Let $f\ \in\ \B\,$, with $s\ =\ 1\,$,
\begin{equation*}
  f\,(x)\ =\ -\,1\ +\ x\ +\ x^{\,n}\ +\ x^{\,m_{\,1}}\,.
\end{equation*}
Let $e_{\,1}\ =\ \gcd\,(m_{\,1},\,n\,-\,1)\,$, $e_{\,2}\ =\ \gcd\,(n,\,m_{\,1}\,-\,1)\,$. The quadrinomial $f\,(x)$ is irreducible over $\Q$ if and only if
\begin{equation*}
  m_{\,1}\ \not\equiv\ 0~(\mod~ 2\,e_{\,1})\,, \qquad n\ \not\equiv\ 0~(\mod~ 2\,\,e_{\,2})\,.
\end{equation*}
\end{theorem}


\section{Noncyclotomic reciprocal factors}
\label{S4}

In this Section we investigate the possible irreducible factors, in the factorization of a polynomial $f\ \in\ \B_{\,n}\,$, with $n$ large enough, which vanish on the lenticular zeroes, or a subcollection of them. In Proposition~\ref{noncycloreciprPart} it is proved that the degrees of the noncyclotomic reciprocal factors (if they exist), and therefore the degrees of such $f\,$, should be fairly large. Proposition~\ref{noncycloreciprPart} does not say that the degrees of the noncyclotomic reciprocal factors are large. For the sake of simplicity, the value $c_{\,n}$ (defining the lenticulus of zeroes of $f$) is taken to be equal to $-\,\lo \kappa\,$.

\begin{proposition}\label{noncycloreciprPart}
If $f\,(x)\ :=\ -\,1\ +\ x\ +\ x^{\,n}\ +\ x^{\,m_{\,1}}\ +\ x^{\,m_{\,2}}\ +\ \ldots\ +\ x^{\,m_{\,s}}\ \in\ \B_{\,n}\,$, $s\ \geq\ 1\,$, $n\ \geq\ 260\,$, admits a reciprocal noncyclotomic factor in its factorization which has a root of modulus $\geq\ 1\ +\ (1\ -\ c)\cdot\bigl(-\,\frac{\lo\kappa}{n}\bigr)\ +\ c\;(\,\theta_{\,n}^{\,-1}\ -\ 1\,)\,$, for some $0\ \leq\ c\ \leq\ 1\,$, then the number $s\ +\ 3$ of its monomials satisfies:
\begin{equation*}
  s\ +\ 3\ \geq\ \Bigl(1\ +\ \frac{1}{n}\;\lo\frac{n^{\,c}}{\kappa^{\,(1\,-\,c)}}\Bigr)^{\,n\,-\,1}\ +\ 1
\end{equation*}
and its degree has the following lower bound
\begin{equation}\label{degreminni}
  m_{\,s}\ =\ \deg\,f\ \geq\ \biggl(\Bigl(1\ +\ \frac{1}{n}\;\lo\frac{n^{\,c}}{\kappa^{\,(1\,-\,c)}}\Bigr)^{\,n\,-\,1}\ -\ 1\biggr)\cdot(n\ -\ 1)\ +\ 1\,.
\end{equation}
\end{proposition}

\begin{proof}
The \textsc{Perron} number $\theta_{\,n}^{\,-1}$ is the dominant root of $-\,1\ +\ x\ +\ x^{\,n}\,$, and $\theta_{\,n\,-\,1}^{\,-1}$ of $-\,1\ +\ x\ +\ x^{\,n\,-\,1}\,$. Since $f\ \in\ \B_{\,n}\,$, $s\ \geq\ 1\,$, by Lemma~5.1 (ii) in \cite{Flatto1994} (\cf Section~\ref{S5.5}), the dominant (positive real) zero of $f^{\,*}\,(x)$ lies in the interval $(\theta_{\,n}^{\,-1},\; \theta_{\,n\,-\,1}^{\,-1})\,$. The (external) lenticulus of zeroes of $f^{\,*}\,(x)$ is defined as the image of that of $f$ by $z\ \longmapsto\ 1/z\,$. The existence of $c\ \in\ [\,0,\,1\,]$ and a reciprocal noncylotomic factor vanishing at the zeroes of the subcollection of the lenticulus of $f$ defined by $c\,$, implies that this reciprocal noncylotomic factor also vanishes at the zeroes of the lenticulus of $f^{\,*}\,$, external to the unit disk, in the same proportion.

\begin{lemma}[Mignotte -- \c{S}tef\u{a}nescu \cite{Mignotte1999}]\label{PROPmignotte}
Let
\begin{equation*}
  P\,(x)\ =\ x^{\,q}\ +\ a_{\,q\,-\,k}\,x^{\,q\,-\,k}\ +\ \ldots\ +\ a_{\,1}\,x\ +\ a_{\,0}\ \in\ \Z\,[x]\ \setminus\ \Z\,.
\end{equation*}
Then the moduli of the roots of $P\,(x)$ are bounded by
\begin{equation}\label{mignottestef}
  \bigl(\,\abs{a_{\,0}}\ +\ \abs{a_{\,1}}\ +\ \ldots\ +\ \abs{a_{\,q\,-\,k}}\,\bigr)^{1/k}\,.
\end{equation}
\end{lemma}

The number of monomials in $f\ \in\ \B_{\,n}\,$, $n\ \geq\ 2\,$, is equal to $s\ +\ 3\,$. Then the sum $\abs{a_{\,0}}\ +\ \abs{a_{\,1}}\ +\ \ldots\ +\ \abs{a_{\,q\,-\,k}}$ of Proposition~\ref{PROPmignotte}, applied to $P\,(x)\ \equiv\ f\,(x)$ with $q\ =\ m_{\,s}\,$, is equal to $s\ +\ 2\,$, and $k$ is $\geq\ n\ -\ 1\,$. If we assume that $f$ contains an irreducible reciprocal noncyclotomic factor $B$ having a root of modulus $\geq\ 1\ +\ (1\ -\ c)\cdot\Bigl(-\,\frac{\lo\kappa} {n}\Bigr)\ +\ c\cdot(\theta_{\,n}^{\,-\,1}\ -\ 1)\,$, for some $0\ \leq\ c\ \leq\ 1$ then we should have, by Lemma~\ref{PROPmignotte} and by Equation~\eqref{eq:remark1},
\begin{equation*}
  (s\ +\ 2)^{\,1/k}\ \geq\ 1\ +\ \frac{1}{n}\;\Bigl(-\,\lo\kappa^{\,(1\,-\,c)}\ +\ c\,\lo n\,\Bigr)\,.
\end{equation*}
Therefore,
\begin{equation*}
  \frac{1}{k}\;\lo(s\ +\ 2)\ \geq\ \lo\Bigl(1\ +\ \frac{1}{n}\;\lo\frac{n^{\,c}}{\kappa^{\,(1\,-\,c)}}\Bigr)\,,
\end{equation*}
which implies
\begin{equation*}
  \lo(s\ +\ 2)\ \geq\ \lo\biggl(\Bigl(1\ +\ \frac{1}{n}\;\lo\frac{n^{\,c}}{\kappa^{\,(1\,-\,c)}}\Bigr)^{\,n\,-\,1}\biggr)
\end{equation*}
and the result. Moreover,
\begin{multline*}
  m_{\,s}\ =\ (m_{\,s}\ -\ m_{\,s\,-\,1})\ +\ (m_{\,s\,-\,1}\ -\ m_{\,s\,-\,2})\ +\ \ldots\ +\ (m_{\,2}\ -\ m_{\,1})\ +\ (m_{\,1}\ -\ n)\ +\\
  (n\ -\ 1)\ +\ 1\ \geq\ (s\ +\ 1)\cdot(n\ -\ 1)\ +\ 1\,,
\end{multline*}
from which \eqref{degreminni} is deduced.
\end{proof}

\begin{example}
Let $f\ \in\ \B_{\,n}\,$, with $n\ =\ 400\,$, for which it is assumed that there exists a reciprocal noncyclotomic factor of $f$ vanishing on the subcollection of roots of the lenticulus of $f$ given by $c\ =\ 0.95\,$. Then, by \eqref{degreminni}, the degree $m_{\,s}$ of $f$ should be above $121\,786\,$. 
\end{example}

The case where the summit (real $> 1$) of the lenticulus of zeroes of $f^{\,*}$ is a zero of a reciprocal noncyclotomic factor of $f$ never occurs by the following Proposition.

\begin{proposition}\label{descartesrule}
If $f\,(x)\ :=\ -\,1\ +\ x\ +\ x^{\,n}\ +\ x^{\,m_{\,1}}\ +\ x^{\,m_{\,2}}\ +\ \ldots\ +\ x^{\,m_{\,s}}\ \in\ \B_{\,n}\,$, $s\ \geq\ 1\,$, $n\ \geq\ 3\,$, is factorized as $f\,(x)\ =\ A\,(x)\cdot B\,(x)\cdot C\,(x)$ as in Theorem~\ref{thm1factorization}. Then, the unique positive real root of $f\,(x)$ is a root of the nonreciprocal part $C\,(x)\,$.
\end{proposition}

\begin{proof}
By \textsc{Descartes}'\,s rule the number of positive real roots of $f$ should be less than the number of sign changes in the sequence of coefficients of the polynomials $f\,$. The number of sign changes in $f$ is $1\,$. If say $1/\beta$ is the unique root of $f$ in $(0,\,1)\,$, and assumed to be a root of a factor of $B$ then $\beta$ and $1/\beta\ \neq\ \beta$ would be two real roots of $f\,$, what is impossible.
\end{proof}


\section[Lenticuli of zeroes: an example with $s\ =\ 12\,$, and various pentanomials]{Lenticuli of zeroes: an example with $s\ =\ 12\,$, and various pentanomials (with $s\ =\ 2$)}
\label{S3}

In this paragraph let us examplify the fact that the roots of any
\begin{equation*}
  f\,(x)\ :=\ -\,1\ +\ x\ +\ x^{\,n}\ +\ x^{\,m_{\,1}}\ +\ x^{\,m_{\,2}}\ +\ \ldots\ +\ x^{\,m_{\,s}}\,,
\end{equation*}
where $s\ \geq\ 1\,$, $m_{\,1}\ -\ n\ \geq\ n\ -\ 1\,$, $m_{\,q\,+\,1}\ -\ m_{\,q}\ \geq\ n\ -\ 1$ for $1\ \leq\ q\ <\ s\,$, $n\ \geq\ 3\,$, are separated into two parts, those which lie in a narrow annular neighbourhood of the unit circle, and those forming a lenticulus of roots $\omega$ inside an angular sector $-\,\gamma\ <\ \arg\,\omega\ <\ \gamma$ with $\gamma$ say $<\ \pi/3$ off the unit circle. This dichotomy phenomenon becomes particularly visible when $n$ and $s$ are large. This lenticulus is shown to be a deformation of the lenticulus determined by the trinomial $-\,1\ +\ x\ +\ x^{\,n}$ made of the first three terms of $f\,$; the lenticulus of zeroes of $-\,1\ +\ x\ +\ x^{\,n}$ is constituted by the zeroes of real part $>\ 1/2\,$, equivalently which lie in the angular sector $-\,\pi/3\ <\ \arg\,(z)\ <\ \pi/3\,$, symmetrically with respect to the real axis, for which the number of roots is equal to $1\ +\ 2\,\lfloor\,n/6\,\rfloor$ \cite[Prop.~3.7]{Verger-Gaugry2016}.

The value of $\gamma$ is taken equal to $\pi/18$ as soon as $n$ is large enough, due to the structure of the asymptotic expansions of the roots of $G_{\,n}$ \cite{Verger-Gaugry2016}, so that the number of roots of the lenticulus of roots of $f$ can be asymptotically defined by the formula
\begin{equation}\label{nombreasumptoLENTICULE}
  1\ +\ \biggl\lfloor\,\frac{1}{3}\,\Bigl\lfloor\,\frac{n}{6}\,\Bigr\rfloor\,\biggr\rfloor\ \pm\ 1\,.
\end{equation}
At small values of $n\,$, the value of $\gamma\ =\ \pi/18$ is also kept as a critical threshold to estimate the number of elements in the lenticulus of roots of $f$ by \eqref{nombreasumptoLENTICULE}. It can be shown \cite{Verger-Gaugry2018} that the lenticulus of roots of $f$ is a set of zeroes of the nonreciprocal irreducible factor in the factorization of $f\,$. Even though it seems reasonable to expect many roots of $f$ on the unit circle, it is not the case: all the roots $\alpha$ of the nonreciprocal irreducible component of $f\,(x)$ are never on the unit circle: $\abs{\alpha}\ \neq\ 1\,$, as proved in Proposition~\ref{neverunitcircle}.


\bigskip
\paragraph{(i) Example of a polynomial in $\B_{\,37}$ with $s\ =\ 12$}:

Let the polynomial $f\,(x)$ de defined as
\begin{multline}\label{649G20}
  f\,(x)\ :=\ -1\ +\ x\ +\ x^{\,37}\ +\ x^{\,81}\ +\ x^{\,140}\ +\ x^{\,184}\ +\ x^{\,232}\ +\ x^{\,285}\ +\ x^{\,350}\ +\ x^{\,389}\\ 
  +\ x^{\,450}\ +\ x^{\,590}\ +\ x^{\,649}\ \equiv\ G_{\,37}\,(x)\ +\ x^{\,81}\ +\ x^{\,140}\ +\ x^{\,184}\ +\ x^{\,232}\ +\ x^{\,285}\\ 
  +\ x^{\,350}\ +\ x^{\,389}\ +\ x^{\,450}\ +\ x^{\,590}\ +\ x^{\,649}
\end{multline}
The zeroes are represented in Figure~\ref{example649}(\textit{b}), those of
$G_{\,37}\,(x)\ =\ -\,1\ +\ x\ +\ x^{\,37}$ in Figure~\ref{example649}(\textit{a}). The polynomial $f$ is irreducible. The zeroes of $f\,(x)$ are either lenticular or lie very close to the unit circle. The lenticulus of zeroes of $f$ contains $3$ zeroes, compared to $13$ for the cardinal of the lenticulus of zeroes of the trinomial $G_{\,37}\,(x)\,$. It is obtained by a slight deformation of the restriction of the lenticulus of zeroes of $G_{\,37}\,(x)$ to the angular sector $\abs{\arg\,z}\ <\ \pi/18\,$.

\begin{figure}
  \begin{center}
    \subfigure[]{\includegraphics[width=0.48\textwidth]{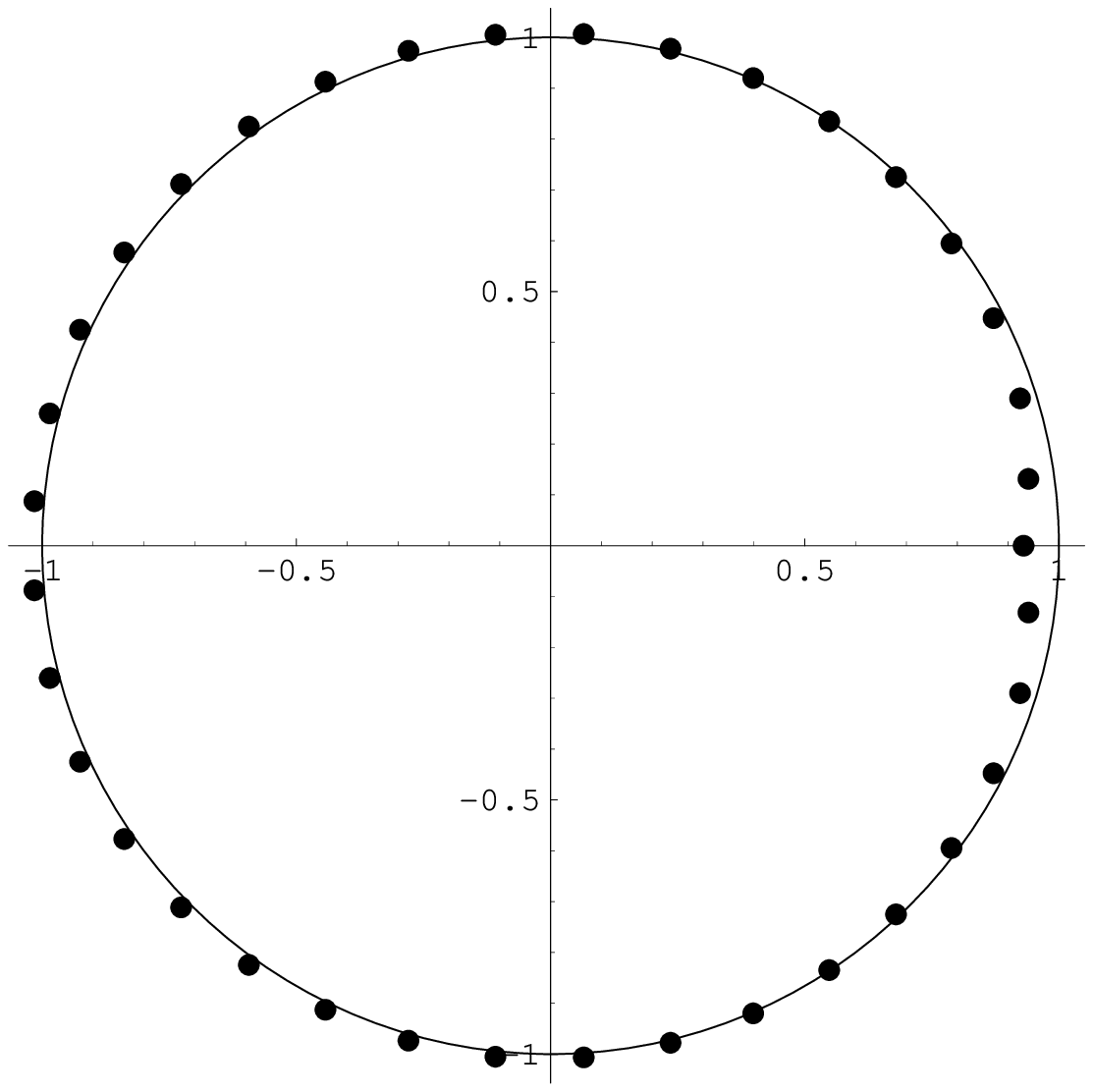}}
    \subfigure[]{\includegraphics[width=0.48\textwidth]{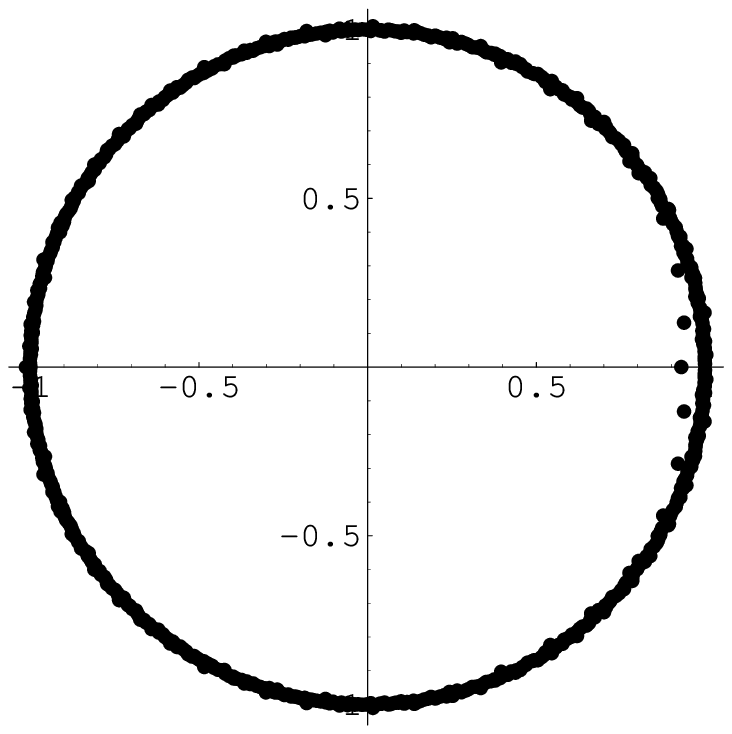}}
  \end{center}
  \caption{\small\em (a) The $37$ zeroes of $G_{37}(x)\ =\ -\,1\ +\ x\ +\ x^{\,37}\,$; (b) The $649$ zeroes of $f\,(x)\ =\ G_{37}\,(x)\ +\ \ldots\ +\ x^{\,649}$ given in Equation~\eqref{649G20}. The lenticulus of roots of $f$ (having $3$ simple zeroes) is obtained by a very slight deformation of the restriction of the lenticulus of roots of $G_{\,37}\,(x)$ to the angular sector $\abs{\arg\,z}\ <\ \pi/18\,$, off the unit circle. The other roots (nonlenticular) of $f$ can be found in a narrow annular neighbourhood of $\abs{z}\ =\ 1\,$.}
  \label{example649}
\end{figure}


\bigskip
\paragraph{(ii) Examples of pentanomials ($s\ =\ 2$)}

The examples show different factorizations of polynomials $f\ \in\ \B_{\,n}$ for various values of $n\,$, having a small number of roots in their lenticulus of roots; in many examples the number of factors is small (one, two or three). The last examples exhibit polynomials $f\ \in\ \B$ having a larger number of zeroes in the lenticuli of roots ($5\,$, $7$ and $27$). Denser lenticuli of roots (for $n\ \geq\ 1\,000$ for instance) are difficult to visualize graphically for the reason that the lenticuli of roots are extremely close to the unit circle, and \emph{apparently} become embedded in the annular neighbourhood of the nonlenticular roots.

\begin{enumerate}
  \item Dynamical degree $n = 5\,$. Let $f_{\,1}\,(x)\ =\ -\,1\ +\ x\ +\ x^{\,5}\ +\ x^{\,9}\ +\ x^{\,15}\,$. It is reducible and its factorization admits only one irreducible cyclotomic factor, the second factor being irreducible nonreciprocal:
  \begin{multline*}
    f_{\,1}\,(x)\ =\ (1\ +\ x\ +\ x^{\,2})\cdot(-\,1\ +\ 2\,x\ -\ x^{\,2}\ -\ x^{\,3}\ +\ 2\,x^{\,4}\ -\ 2\,x^{\,6}\ +\\ 
    2\,x^{\,7}\ -\ x^{\,9}\ +\ x^{\,10}\ -\ x^{\,12}\ +\ x^{\,13})\,.
  \end{multline*}
  Let $f_{\,2}\,(x)\ =\ -\,1\ +\ x\ +\ x^{\,5}\ +\ x^{\,9}\ +\ x^{\,18}\,$. In the factorization of $f_{\,2}\,(x)$ two irreducible cyclotomic factors appear and where the third factor is irreducible and nonreciprocal:
  \begin{multline*}
    f_{\,2}\,(x)\ =\ (1\ -\ x\ +\ x^{\,2})\cdot(1\ +\ x\ +\ x^{\,2})\cdot(-\,1\ +\ x\ +\ x^{\,2}\ -\ x^{\,3}\ +\ x^{\,5}\\
    -\ x^{\,6}\ +\ x^{\,8}\ -\ x^{\,12}\ +\ x^{\,14})\,.
  \end{multline*}
  In both cases, the lenticulus of zeroes of $f_{\,1,\,2}\,(x)$ is the lenticulus of its nonreciprocal factor. It is reduced to the unique real positive zero of $f_{\,1,\,2}\,$: $0.7284\ldots\,$, resp. $0.7301\ldots\,$, close to real positive zero $0.7548\ldots$ of $G_{\,5}\,(x)$ which is the only element of the lenticulus of roots of $G_{\,5}\,(x)\,$.

  \item Dynamical degree $n\ =\ 12\,$. The lenticulus of zeroes of $G_{\,12}\,(x)$ is shown in Figure~\ref{lt05}(\textit{a}) and Figure~\ref{lt02}(\textit{a}). It contains $5$ zeroes. Let $f_{\,1}\,(x)\ =\ -\,1\ +\ x\ +\ x^{\,12}\ +\ x^{\,23}\ +\ x^{\,35}\,$, resp. $f_{\,2}\,(x)\ =\ -\,1\ +\ x\ +\ x^{\,12}\ +\ x^{\,250}\ +\ x^{\,385}\,$. Both polynomials are irreducible. In both cases the lenticulus of zeroes of $f\,(x)$ (Figures~\ref{lt05}(\textit{b}), \ref{lt02}(\textit{b})) only contains one point, the real root $0.8447\ldots\,$, resp. $0.8525\ldots\,$, close to the real positive zero $0.8525\ldots$ of $G_{\,12}\,(x)\,$: the lenticulus of $f$ is a slight deformation of the restriction of the lenticulus of $G_{\,12}\,(x)$ to the angular sector $\abs{\arg\,z}\ <\ \pi/18\,$. Comparing Figure~\ref{lt05}(\textit{b}) and Figure~\ref{lt02}(\textit{b}), the higher degree of $f\,$, $385$ instead of $35\,$, has two consequences:
  \begin{enumerate}
    \item the densification of the annular neighbourhood of $\abs{z}\ =\ 1$ by the zeroes of $f\,(x)\,$,
    \item the decrease of the thickness of the annular neighbourhood containing the nonlenticular roots of $f\,(x)\,$. This phenomenon is general (\cf Section~\ref{S5.4}).
  \end{enumerate}

  \item Dynamical degree $n\ =\ 81\,$. Let $f\,(x)\ =\ -\,1\ +\ x\ +\ x^{\,81}\ +\ x^{\,165}\ +\ x^{\,250}\,$. It is irreducible. The lenticulus of zeroes of $G_{\,81}\,(x)$ contains $27$ points (Figure~\ref{lt06}(\textit{a})), while that of $f\,(x)$ (Figure~\ref{lt06}(\textit{b})) contains $5$ points, in particular the real root $0.9604\ldots\,$, close to the real positive root $0.9608\ldots$ of $G_{\,81}\,(x)\,$.

  \item Dynamical degree $n\ =\ 121\,$. Let $f\,(x)\ =\ -\,1\ +\ x\ +\ x^{\,121}\ +\ x^{\,250}\ +\ x^{\,385}\,$. It is irreducible. The lenticulus of zeroes of $G_{\,121}\,(x)$ contains $41$ points (see Figure~\ref{lt07}(\textit{a})), whereas the lenticulus of roots of $f\,(x)$ (Figure~\ref{lt07}(\textit{b})) contains $7$ points, in particular the real root $0.9709\ldots\,$, close to the real positive root $0.971128\ldots$ of $G_{\,121}\,(x)\,$.

\end{enumerate}

\begin{figure}
  \begin{center}
    \subfigure[]{\includegraphics[width=0.48\textwidth]{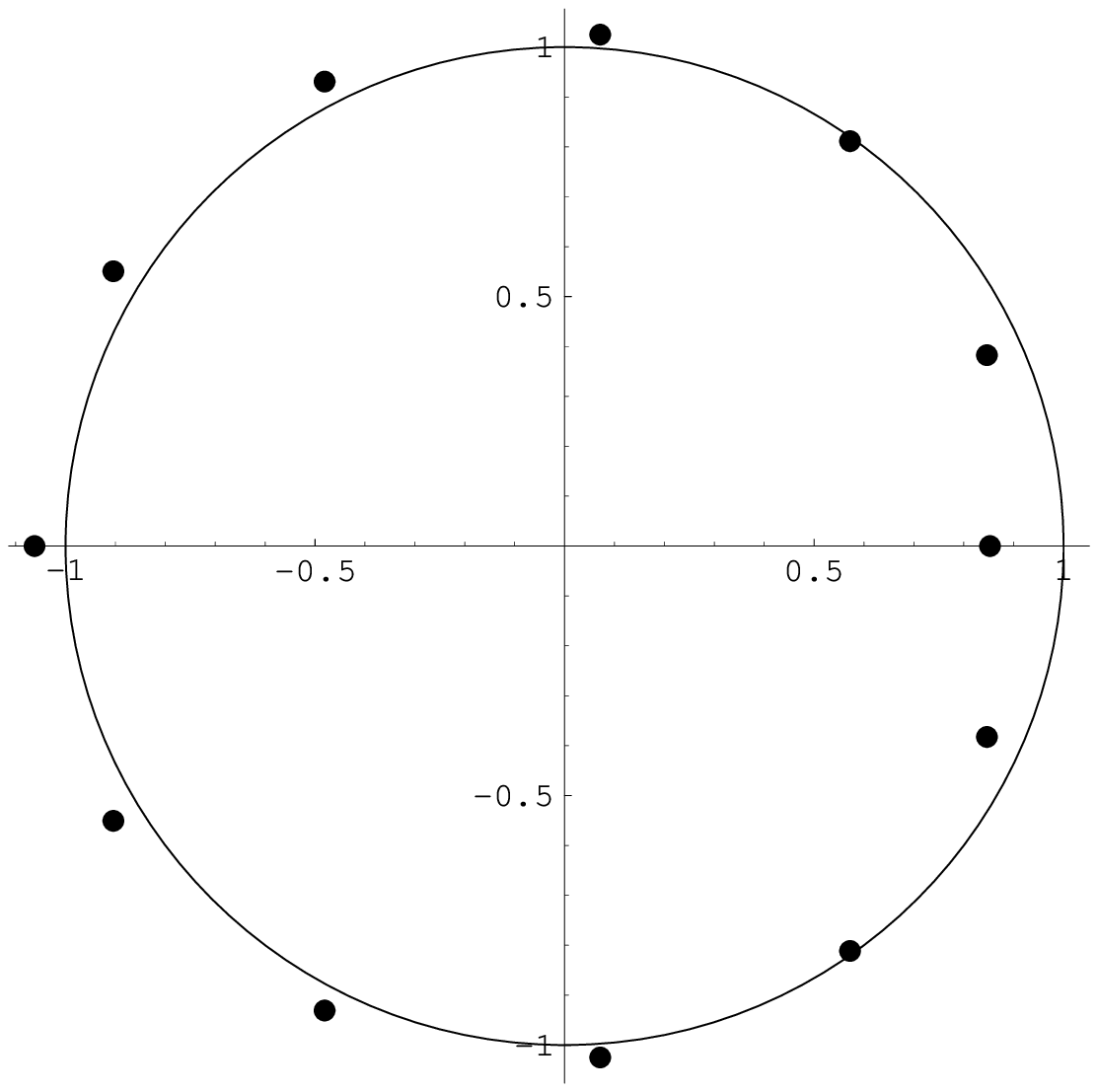}}
    \subfigure[]{\includegraphics[width=0.48\textwidth]{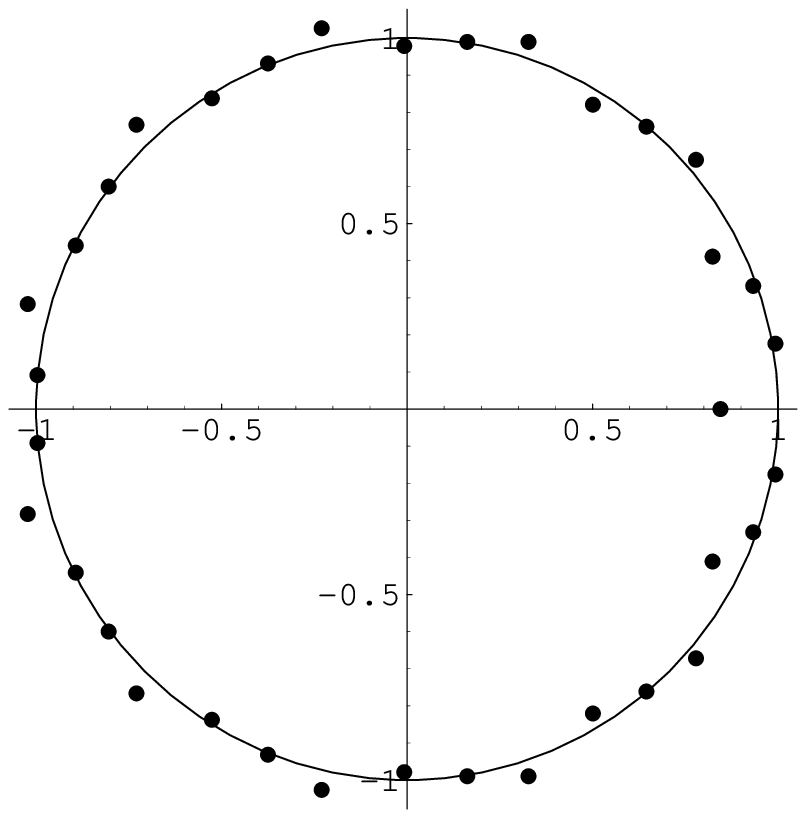}}
  \end{center}
  \caption{\small\em (a) The $12$ zeroes of $G_{\,12}\,(x)\,$; (b) The $35$ simple zeroes of $f\,(x)\ =\ -\,1\ +\ x\ +\ x^{\,12}\ +\ x^{\,23}\ +\ x^{\,35}\,$. By definition, only one root is lenticular, the one on the real axis, though the ``complete'' lenticulus of roots of $-\,1\ +\ x\ +\ x^{\,12}\,$, slightly deformed, can be guessed.}
  \label{lt05}
\end{figure}

\begin{figure}
  \begin{center}
    \subfigure[]{\includegraphics[width=0.48\textwidth]{figs/trinomial_conjg_X12X1m1.eps}}
    \subfigure[]{\includegraphics[width=0.48\textwidth]{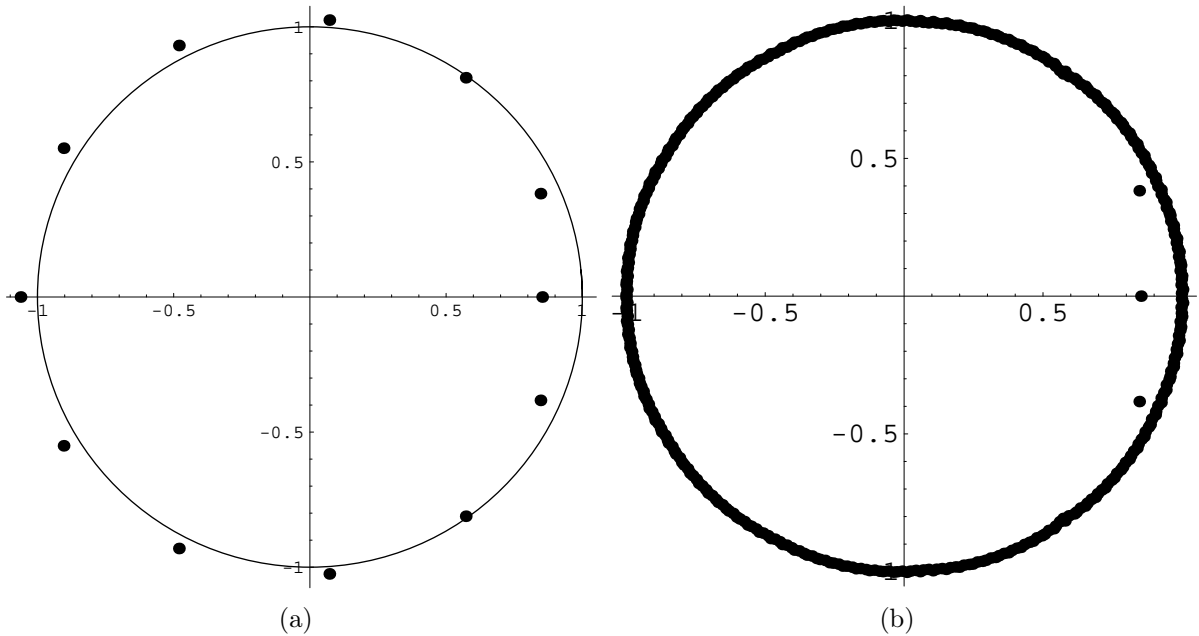}}
  \end{center}
  \caption{\small\em (a) The $12$ zeroes of $G_{\,12}\,(x)\,$; (b) The $385$ zeroes of $f\,(x)\ =\ -\,1\ +\ x\ +\ x^{\,12}\ +\ x^{\,250}\ +\ x^{\,385}\,$. The lenticulus of roots of the trinomial $G_{\,12}\,(x)\ =\ -\,1\ +\ x\ +\ x^{\,12}$ can be guessed, slightly deformed and almost ``complete''. It is well separated from the other roots, and off the unit circle. Only one root of $f$ is considered as a lenticular zero, the one on the real axis: $0.8525\ldots\,$. The thickness of the annular neighbourhood of $\abs{z}\ =\ 1$ which contains the nonlenticular zeroes of $f\,(x)$ is much smaller than in Figure~\ref{lt05}(b).}
  \label{lt02}
\end{figure}

\begin{figure}
  \begin{center}
    \subfigure[]{\includegraphics[width=0.5\textwidth]{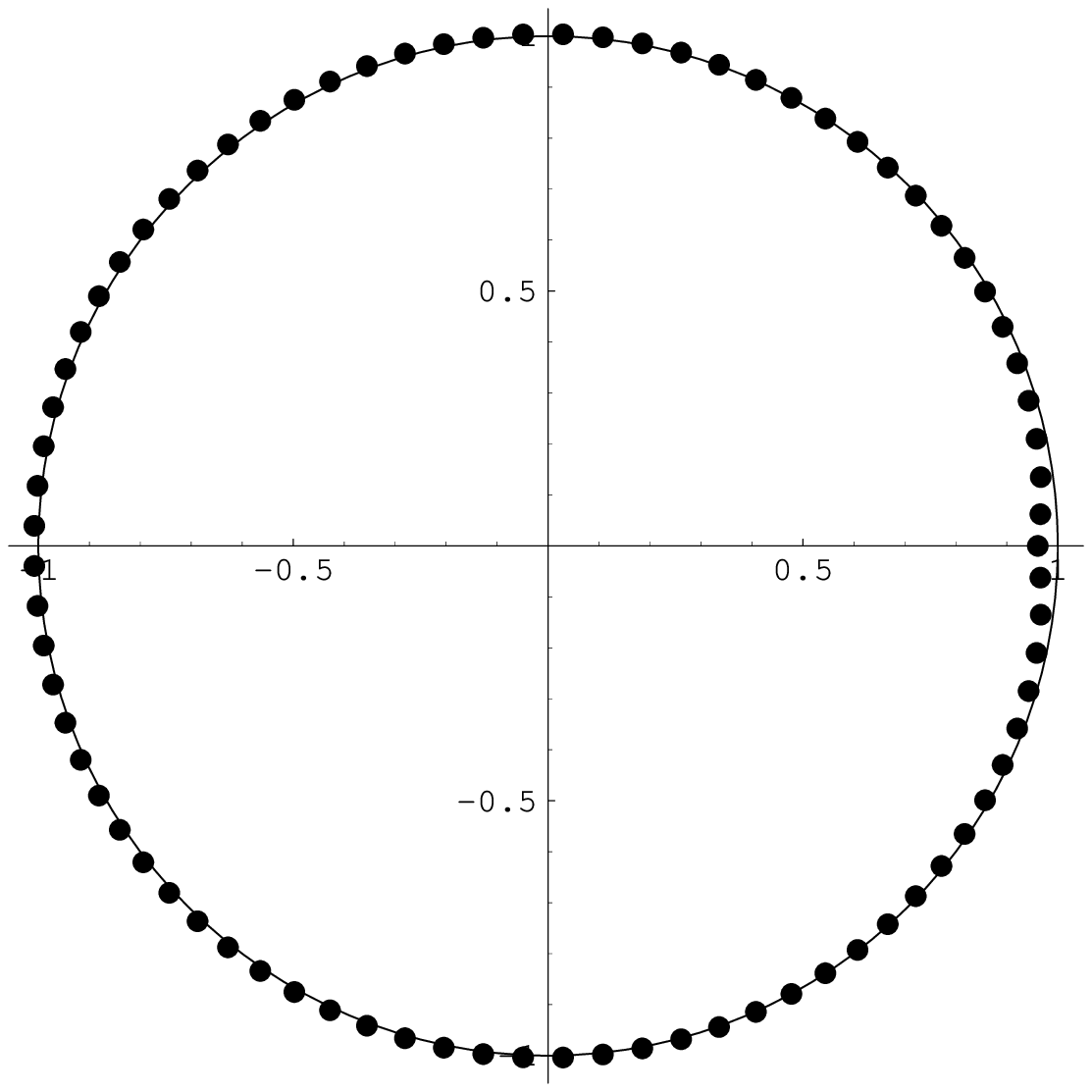}}
    \subfigure[]{\includegraphics[width=0.95\textwidth]{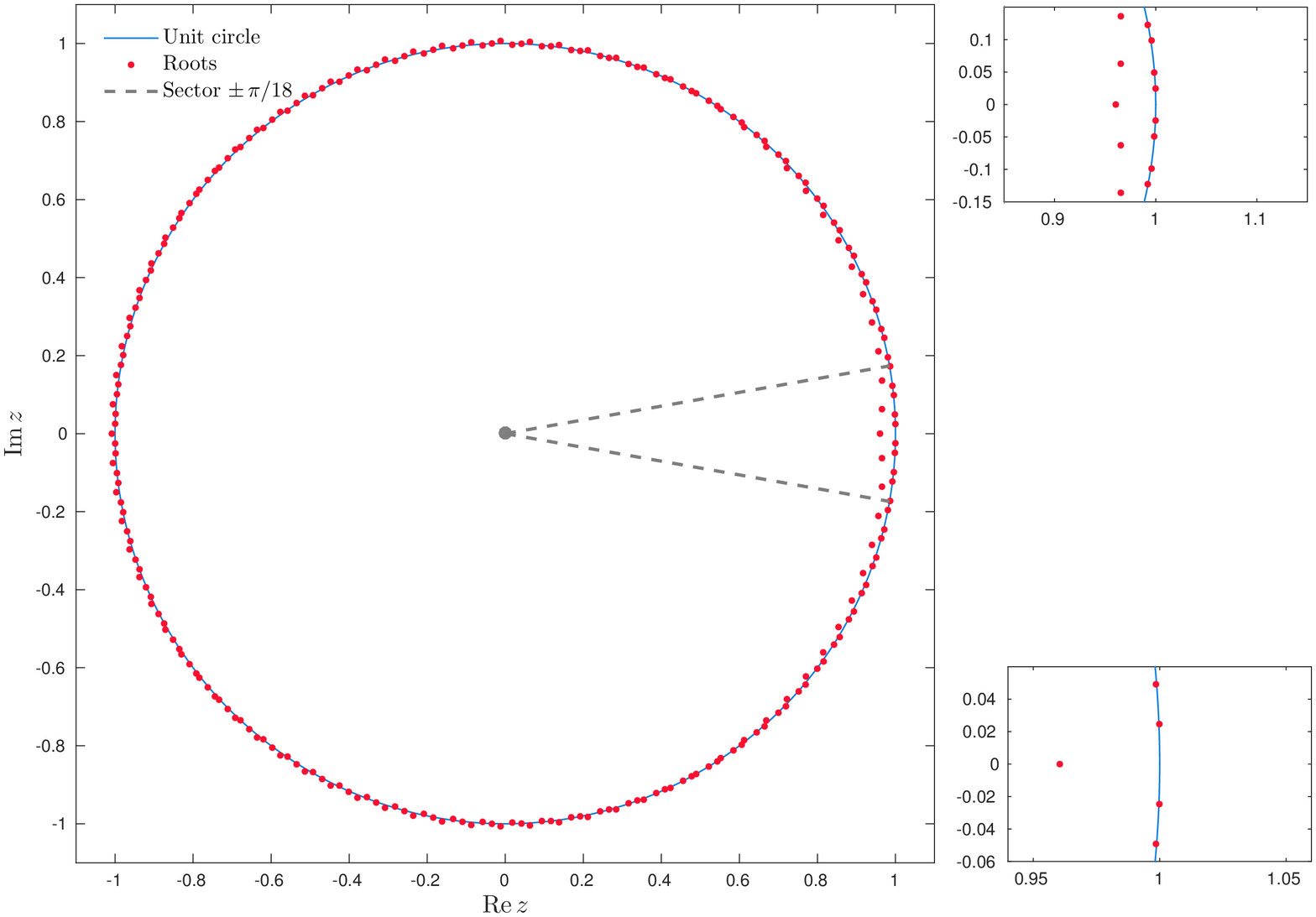}}
  \end{center}
  \caption{\small\em (a) Zeroes of $G_{\,81}\,(x)\,$; (b) Zeroes of $f\,(x)\ =\ -\,1\ +\ x\ +\ x^{\,81}\ +\ x^{\,165}\ +\ x^{\,250}\,$. On the right the distribution of the roots of $f\,(x)$ is zoomed twice in the angular sector $-\,\pi/18\ <\ \arg\,(z)\ <\ \pi/18\,$. The number of lenticular roots of $f\,(x)$ is equal to $5\,$.}
  \label{lt06}
\end{figure}

\begin{figure}
  \begin{center}
    \subfigure[]{\includegraphics[width=0.5\textwidth]{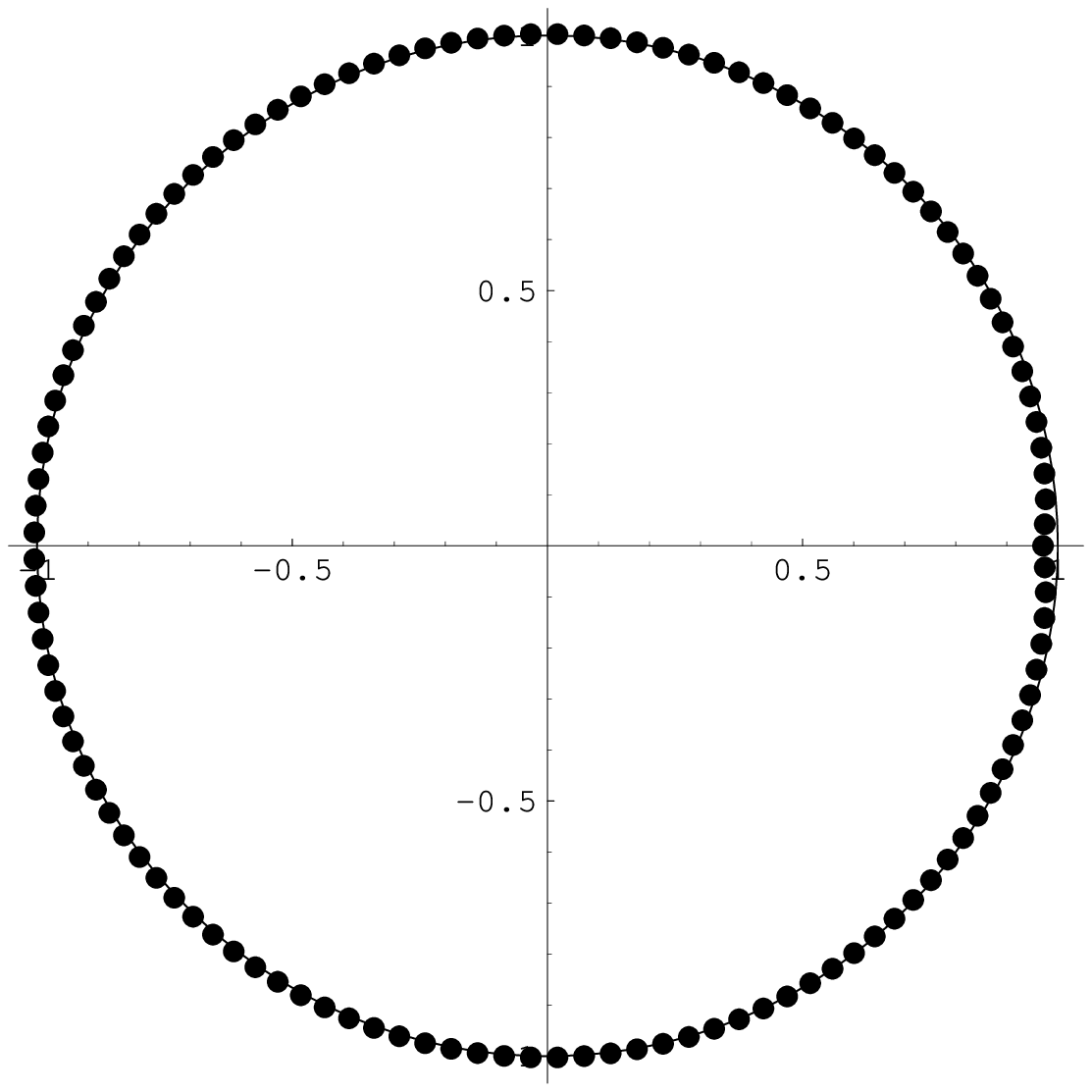}}
    \subfigure[]{\includegraphics[width=0.95\textwidth]{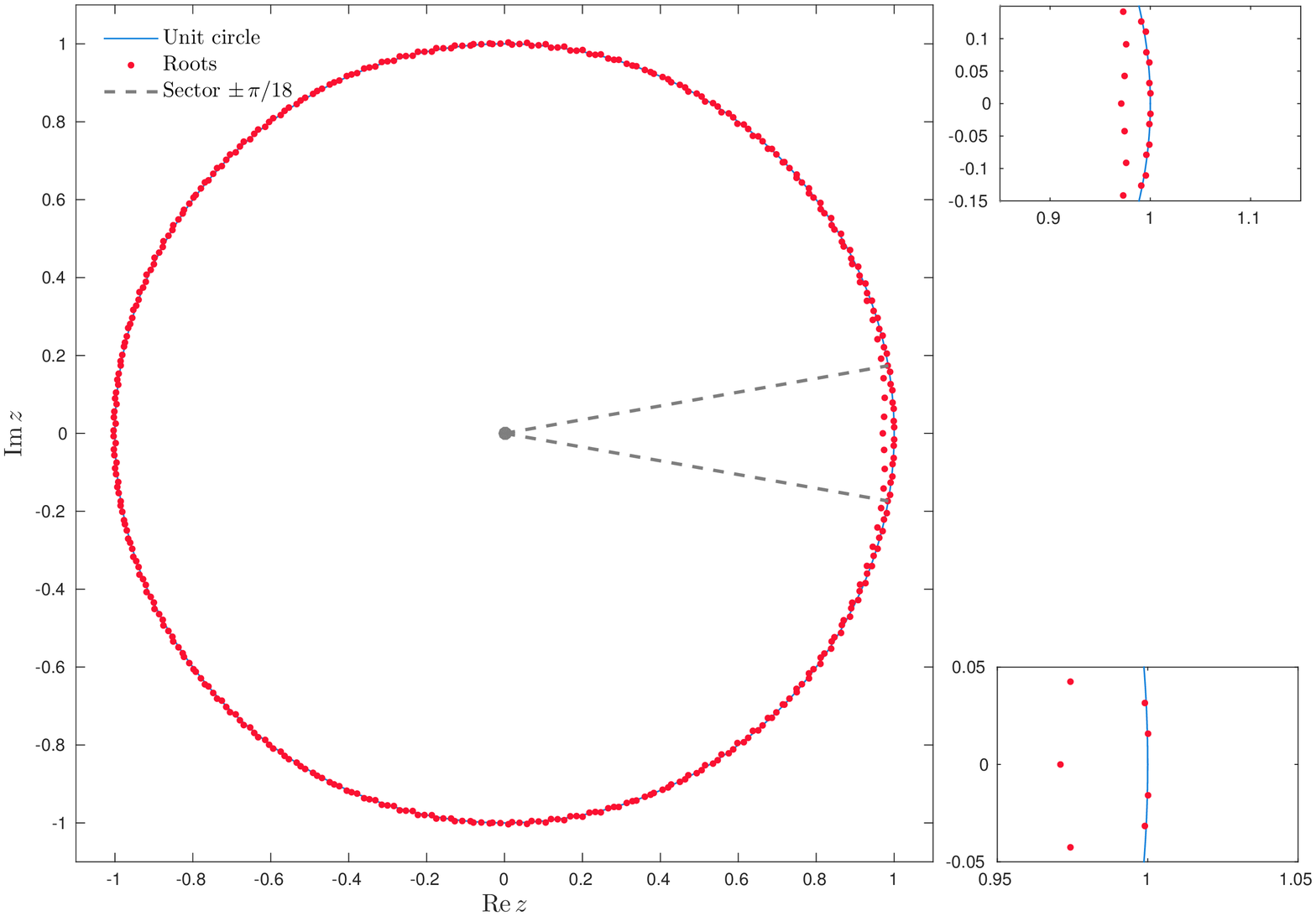}}
  \end{center}
  \caption{\small\em (a) Zeroes of $G_{\,121}\,(x)\,$; (b) Zeroes of $f\,(x)\ =\ -\,1\ +\ x\ +\ x^{\,121}\ +\ x^{\,250}\ +\ x^{\,385}\,$. On the right the distribution of the roots of $f\,(x)$ is zoomed twice in the angular sector $-\,\pi/18\ <\ \arg\,(z)\ <\ \pi/18\,$. The lenticulus of roots of $f\,(x)$ has $7$ zeroes.}
  \label{lt07}
\end{figure}


\section{Factorization of the lacunary polynomials of class $\B$}
\label{S5}

In a series of papers \textsc{Schinzel} \cite{Schinzel1969, Schinzel1978, Schinzel1976, Schinzel1983} has studied the reducibility of lacunary polynomials, their possible factorizations, the asymptotics of their numbers of irreducible factors, reciprocal, nonreciprocal, counted with multiplicities or not, for large degrees. \textsc{Dobrowolsky} \cite{Dobrowolski1979} has also contributed in this domain in view of understanding the problem of \textsc{Lehmer}. First let us deduce the following Theorem on the class $\B\,$, from \textsc{Schinzel}'\,s Theorems.

\begin{theorem}
Suppose $f\ \in\ \B$ of the form
\begin{equation*}
  f\,(x)\ =\ -\,1\ +\ x\ +\ x^{\,n}\ +\ x^{\,m_{\,1}}\ +\ \ldots\ +\ x^{\,m_{\,s}}\,, \quad n\ \geq\ 2\,, \quad s\ \geq\ 1\,.
\end{equation*}
Then the number $\omega\,(f)\,$, resp. $\omega_{\,1}\,(f)\,$, of irreducible factors, resp. of irreducible noncyclotomic factors, of $f\,(x)$ counted without multiplicities in both cases, satisfy
\begin{itemize}
  \item \begin{equation*}
    \omega\,(f)\ \ll\ \sqrt{\frac{m_{\,s}\,\lo(s\,+\,3)}{\lo\lo m_{\,s}}}\,, \qquad m_{\,s}\ \to\ \infty\,,
  \end{equation*}
  \item for every $\epsilon\ \in\ (0,\,1)\,$,
  \begin{equation*}
    \omega_{\,1}\,(f)\ =\ o\,\left(m_{\,s}^{\,\epsilon}\right)\cdot\bigl(\lo(s\ +\ 3)\bigr)^{\,1\,-\,\epsilon}\,, \qquad m_{\,s}\ \to\ \infty\,.
  \end{equation*}
\end{itemize}
\end{theorem}

\begin{proof}
Theorem~1 and Theorem~2, with the ``\emph{Note added in proof}'' in \textsc{Schinzel} \cite[p.~319]{Schinzel1983}.
\end{proof}


\subsection{Cyclotomic parts}
\label{S5.1}

Let us first mention some results on the existence of cyclotomic factors in the factorization of the polynomials of the class $\B\,$. Then, in Proposition~\ref{beaucoup}, we prove the existence of infinitely many polynomials $f\ \in\ \B$ which are divisible by a given cyclotomic polynomial $\Phi_{\,p}\,$, for every prime number $p\ \geq\ 3\,$.

\begin{lemma}\label{lemmaexiste}
Suppose $f\ \in\ \B$ of the form
\begin{equation*}
  f\,(x)\ =\ -\,1\ +\ x\ +\ x^{\,n}\ +\ x^{\,m_{\,1}}\ +\ \ldots\ +\ x^{\,m_{\,s}}\,, \quad n\ \geq\ 2\,, \quad s\ \geq\ 1\,,
\end{equation*}
and divisible by a cyclotomic polynomial. Then there is an integer $m\ =\ p_{\,1}^{\,q_{\,1}}\cdot\, \ldots\,\cdot p_{\,r}^{\,q_{\,r}}$ having all its prime factors $p_{\,i}\ \leq\ s\ +\ 3$ such that $\Phi_{\,m}\,(x)$ divides $f\,(x)\,$.
\end{lemma}

\begin{proof}
Lemma~3.2 in \cite{Filaseta2006}.
\end{proof}

The divisibility of $f\ \in\ \B$ by cyclotomic polynomials $\Phi_{\,p}\,(x)\,$, where $p$ are prime numbers, implies a condition on those $p\,$'\,s by the following Proposition~\ref{prop5.4}.

\begin{lemma}[Boyd]\label{boydlemma}
Let $p$ be a prime number. Suppose $f\ \in\ \B$ of the form
\begin{equation*}
  f\,(x)\ =\ \sum_{\,j\,=\,0}^{\,m_{\,s}} a_{\,j}\,x^{\,j}\ =\ -\,1\ +\ x\ +\ x^{\,n}\ +\ x^{\,m_{\,1}}\ +\ \ldots\ +\ x^{\,m_{\,s}}\,, \quad n\ \geq\ 2\,, \quad s\ \geq\ 1\,.
\end{equation*}
Denote $c_{\,i}\ =\ \sum_{\,k\,\equiv\,i\,(p)} a_{\,k}\,$. Then,
\begin{equation*}
  \Phi_{\,p}\,(x)\ \bigl\vert\ f(x)\ \Longleftrightarrow\ c_{\,0}\ =\ c_{\,1}\ =\ \ldots\ =\ c_{\,p\,-\,1}\,.
\end{equation*}
\end{lemma}

\begin{proof}
$\Phi_{\,p}\,(x)$ divides $f\,(x)$ if and only if $(x^{\,p}\ -\ 1)$ divides $(x\ -\ 1)\cdot f\,(x)\,$.
\end{proof}

\begin{proposition}\label{prop5.4}
Suppose $f(x) \in \mathcal{B}$ of the form
\begin{equation*}
  f\,(x)\ =\ \sum_{\,j\,=\,0}^{\,m_{\,s}}\, a_{\,j}\,x^{\,j}\ =\ -\,1\ +\ x\ +\ x^{\,n}\ +\ x^{\,m_{\,1}}\ +\ \ldots\ +\ x^{\,m_{\,s}}\,, \quad n\ \geq\ 2\,, \quad s\ \geq\ 1
\end{equation*}
and that $\Phi_{\,p}\,(x)\ \bigl\vert\ f\,(x)$ for some prime number $p\,$. Then
\begin{equation*}
  p\ \bigl\vert\ (s\ +\ 1)\,.
\end{equation*}
\end{proposition}

\begin{proof}
Using Lemma~\ref{boydlemma}, since $f\,(1)\ =\ s\ +\ 1\ =\ \sum_{\,k}\, a_{\,k}\ =\ \sum_{\,i\,=\,0}^{p\,-\,1}\,\sum_{\,k\,\equiv\,i\,(p)} a_{\,k}\ =\ p \cdot c_{\,0}\,$, we deduce the claim.
\end{proof}

A necessary condition for $f\,(x)$ to be divisible by $\Phi_{\,p}\,(x)$ is that $s$ should be congruent to $-1$ modulo $p\,$.

\begin{proposition}\label{beaucoup}
Let $p\ \geq\ 3$ be a prime number. Let $n\ \geq\ 2\,$. There exist infinitely many $f\ \in\ \B_{\,n}$ such that
\begin{equation*}
  \Phi_{\,p}\,(x)\ \bigl\vert\ f\,(x)\,.
\end{equation*}
\end{proposition}

\begin{proof}
Let $\zeta_{\,p}$ denote the primitive root of unity $\ue^{\,2\,\ui\,\pi\,/\,p}\,$. Let us assume that $f\,(x)\ \in\ \B_{\,n}$ vanishes at $\zeta_{\,p}\,$, as
\begin{equation*}
  f\,(\zeta_{\,p})\ =\ -\,1\ +\ \zeta_{\,p}\ +\ \zeta_{\,p}^{\,n}\ +\ \zeta_{\,p}^{\,m_{\,1}}\ +\ \ldots\ +\ \zeta_{\,p}^{\,m_{\,s}}\ =\ 0\,.
\end{equation*}
We consider the residues modulo $p$ of the $(s\ +\ 1)-$tuple $(n,\; m_{\,1},\; m_{\,2},\;\ldots,\;m_{\,s})$ so that $f\,(\zeta_{\,p})$ can be written
\begin{multline*}
  f\,(\zeta_{\,p})\ =\ c_{\,0}\ +\ c_{\,1}\,\zeta_{\,p}\ +\ c_{\,2}\,\zeta_{\,p}^{\,2}\ +\ \ldots\ +\ c_{\,p\,-\,1}\,\zeta_{\,p}^{\,p\,-\,1}\ =\ 0\,, \\ 
  c_{\,0},\, c_{\,1},\, \ldots,\, c_{\,p\,-\,1}\ \in\ \Z\,.
\end{multline*}
The polynomial $\Phi_{\,p}\,(x)\ =\ 1\ +\ x\ +\ x^{\,2}\ +\ \ldots\ +\ x^{\,p\,-\,1}\ =\ (x^{\,p}\ -\ 1)\,/\,(x\ -\ 1)$ is the minimal polynomial of $\zeta_{\,p}\,$. Then, if $c_{\,0}$ or $c_{\,p\,-\,1}$ is equal to $0\,$, then all the coefficients $c_{\,i}$ should be equal to $0$ since $\{1,\;\zeta_{\,p},\;\zeta_{\,p}^{\,2},\;\ldots,\;\zeta_{\,p}^{\,p\,-\,2}\}$ is a free system over $\Z\,$. If $c_{\,0}\cdot c_{\,p\,-\,1}\ \neq\ 0$ then the equalities
\begin{equation*}
  c_{\,0}\ =\ c_{\,1}\ =\ c_{\,2}\ =\ \ldots\ =\ c_{\,p\,-\,1}\ \neq\ 0
\end{equation*}
should hold since the polynomial $\sum_{\,j\,=\,0}^{\,p\,-\,1}\,c_{\,j}\,x^{\,j}$ vanishes at $\zeta_{\,p}$ and is of the same degree as $\Phi_{\,p}\,(x)\,$. The common value can be arbitrarily large. In both cases we have the condition
\begin{equation*}
  c_{\,0}\ =\ c_{\,1}\ =\ c_{\,2}\ =\ \ldots\ =\ c_{\,p\,-\,1}\,.
\end{equation*}
It means that the distribution of the exponents $n\,$, $m_{\,1}\,$, $m_{\,2}\,$, $\ldots$, $m_{\,s}$ by class of congruence modulo $p$ should be identical in each class.

Then, if $p\ \leq\ n\,$, the constant term $-\,1$ ``belongs to'' the class ``$\ \equiv\ 0\ \mod\ p\,$'', and $\zeta_{\,p}$ to the class ``$\ \equiv\ 1\ \mod\ p\,$''. The term $\zeta_{\,p}^{\,n}$ may belong to another class ``$\ \equiv i\ \mod\ p\,$'' with $i\ \neq\ 0$, $1$ or to one of the classes ``$\ \equiv 0\ \mod p\,$'' or ``$\ \equiv\ 1\ \mod\ p\,$''. If $p\ >\ n$ then the term $\zeta_{\,p}^{\,n}$ belongs to another class ``$\ \equiv\ i\ \mod\ p$'' with $i\ \neq\ 0\,$, $1\,$. In both cases we can complete the classes by suitably adding terms ``$\zeta_{\,p}^{\,m_{\,i}}\,$''. We now chose $s\ \geq\ 1$ and $m_{\,1}\,$, $m_{\,2}\,$, $\ldots$, $m_{\,s}$ sequentially such that the distribution of the residues modulo $p$ 
\begin{equation*}
  m_{\,1} \mod\ p\,, \quad m_{\,2} \mod\ p\,, \quad \ldots, \quad m_{\,s} \mod \ p
\end{equation*}
in the respective classes ``$\ \equiv\ i \mod\ p\,$", with $i\ =\ 0,\, 1,\, \ldots,\, p\,-\,1\,$, is equal.

If one solution $(m_{\,1},\,\ldots,\,m_{\,s})$ is found, then $\Phi_{\,p}\,(x)$ divides $f\,(x)\,$. Another solution $f^{\,\sharp}\ \in\ \B_{\,n}$ is now found with $s^{\,\sharp}\ =\ s\ +\ p$ and a suitable choice of the exponents $m_{\,s\,+\,1},\,\ldots,\,m_{\,s\,+\,p}$
\begin{equation*}
  f^{\,\sharp}\,(x)\ =\ -\,1\ +\ x\ +\ x^{\,n}\ +\ x^{\,m_{\,1}}\ +\ \ldots\ +\ x^{\,m_{\,s}}\ +\ x^{\,m_{\,s\,+\,1}}\ +\ \ldots\ x^{\,m_{\,s\,+\,p}}
\end{equation*}
so that
\begin{equation*}
  f^{\,\sharp}\,(\zeta_{\,p})\ =\ c^{\,\sharp}_{\,0} + c^{\,\sharp}_{\,1}\,\zeta_{\,p}\ +\ \ldots\ +\ c^{\,\sharp}_{\,p\,-\,1}\,\zeta_{\,p}^{\,p\,-\,1}\,,
\end{equation*}
where the $p$ residues modulo $p$ of $m_{\,s\,+\,1},\,\ldots,\,m_{\,s\,+\,p}$ are all distinct, satisfying
\begin{equation*}
  c^{\,\sharp}_{\,0}\ =\ c^{\,\sharp}_{\,1}\ =\ \ldots\ =\ c^{\,\sharp}_{\,p\,-\,1}\ =\ c_{\,0}\ +\ 1\,.
\end{equation*}
Then $\Phi_{\,p}\,(x)$ also divides $f^{\,\sharp}\,(x)\,$. Iterating this process we deduce the claim.
\end{proof}


\subsection{Nonreciprocal parts}
\label{S5.3}

\begin{proposition}\label{neverunitcircle}
If $P\,(z)\ \in\ \Z\,[z]\,$, $P\,(1)\ \neq\ 0\,$, is nonreciprocal and irreducible, then $P\,(z)$ has no root of modulus $1\,$.
\end{proposition}

\begin{proof}
Let $P\,(z)\ =\ a_{\,d}\,z^{\,d}\ +\ \ldots\ +\ a_{\,1}\,z\ +\ a_{\,0}\,$, with $a_{\,0}\cdot a_{\,d}\ \neq\ 0\,$, be irreducible and nonreciprocal. We have $\gcd\,(a_{\,0},\,\ldots,\,a_{\,d})\ =\ 1\,$. If $P\,(\zeta)\ =\ 0$ for some $\zeta\,$, $\abs{\zeta}\ =\ 1\,$, then $P\,(\bar{\zeta})\ =\ 0\,$. But $\bar{\zeta}\ =\ 1/\zeta$ and then $P\,(z)$ would vanish at $1/\zeta\,$. Hence $P$ would be a multiple of the minimal polynomial $P^{\,*}$ of $1/\zeta\,$. Since $\deg\,(P)\ =\ \deg(P^{\,*})$ there exists $\lambda\ \neq\ 0\,$, $\lambda\ \in\ \Q\,$, such that $P\ =\ \lambda\,P^{\,*}\,$. In particular, looking at the dominant
and constant terms, $a_{\,0}\ =\ \lambda\,a_{\,d}$ and $a_{\,d}\ =\ \lambda\,a_{\,0}\,$. Hence, $a_{\,0}\ =\ \lambda^{\,2}\,a_{\,0}\,$, implying $\lambda\ =\ \pm\,\,\,1\,$. Therefore $P^{\,*}\ =\ \pm\,P\,$. Since $P$ is assumed nonreciprocal, $P^{\,*}\ \neq\ P\,$, implying $P^{\,*}\ =\ -\,P\,$. Since $P^{\,*}\,(1)\ =\ P\,(1)\ =\ -\,P\,(1)\,$, we would have $P\,(1)\ =\ 0\,$. Contradiction.
\end{proof}

For studying the irreducibility of the nonreciprocal parts of the polynomials $f\ \in\ \B\,$, we will follow the method introduced by \textsc{Ljunggren} \cite{Ljunggren1960}, used by \textsc{Schinzel} \cite{Schinzel1969, Schinzel1978} and \textsc{Filaseta} \cite{Filaseta1999}.

\begin{lemma}[Ljunggren \cite{Ljunggren1960}]\label{reducibi}
Let $P\,(x)\ \in\ \Z\,[\,x\,]\,$, $\deg\,(P)\ \geq\ 2\,$, $P\,(0)\ \neq\ 0\,$. The nonreciprocal part of $P\,(x)$ is reducible if and only if there exists $w\,(x)\ \in\ \Z\,[\,x\,]$ different from $\pm\,P\,(x)$ and $\pm\,P^{\,*}\,(x)$ such that $w\,(x)\cdot w^{\,*}\,(x)\ =\ P\,(x)\cdot P^{\,*}\,(x)\,$.
\end{lemma}

\begin{proof}
Let us assume that the nonreciprocal part of $P\,(x)$ is reducible. Then, there exists two nonreciprocal polynomials $u\,(x)$ and $v\,(x)$ such that $P\,(x)\ =\ u\,(x)\cdot v\,(x)\,$. Let $w\,(x)\ =\ u\,(x)\cdot v^{\,*}\,(x)\,$. We have:
\begin{equation*}
  w\,(x)\cdot w^{\,*}\,(x)\ =\ u\,(x)\cdot v^{\,*}\,(x)\cdot u^{\,*}\,(x)\cdot v\,(x)\ =\ P\,(x)\cdot P^{\,*}\,(x)\,.
\end{equation*}
Conversely, let us assume that the nonreciprocal part $c\,(x)$ of $P\,(x)$ is irreducible and that there exists $w\,(x)$ different of $\pm\,P\,(x)$ and $\pm\,P^{\,*}\,(x)$ such that $w\,(x)\cdot w^{\,*}\,(x)\ =\ P\,(x)\cdot P^{\,*}\,(x)\,$. Let $P\,(x)\ =\ a\,(x)\cdot c\,(x)$ be the factorization of $P$ where every irreducible factor in $a$ is reciprocal. Then,
\begin{equation*}
  P\,(x)\cdot P^{\,*}\,(x)\ =\ a^{\,2}\,(x)\cdot c\,(x)\cdot c^{\,*}\,(x)\ =\ w\,(x)\cdot w^{\,*}\,(x)\,.
\end{equation*}
We deduce $w\,(x)\ =\ \pm\,a\,(x)\cdot c\,(x)\ =\ \pm\,P\,(x)$ or $w\,(x)\ =\ \pm \,a\,(x)\cdot c^{\,*}\,(x)\ =\ \pm\,P^{\,*}\,(x)\,$. Contradiction.
\end{proof}

\begin{proposition}\label{nonreciprocalpart}
For any $f\ \in\ \B_{\,n}\,$, $n\ \geq\ 3\,$, denote by
\begin{equation*}
  f\,(x)\ =\ A\,(x)\cdot B\,(x)\cdot C\,(x)\ =\ -\,1\ +\ x\ +\ x^{\,n}\ +\ x^{\,m_{\,1}}\ +\ x^{\,m_{\,2}}\ +\ \ldots\ +\ x^{\,m_{\,s}}\,,
\end{equation*}
where $s\ \geq\ 1\,$, $m_{\,1}\ -\ n\ \geq\ n\ -\ 1\,$, $m_{\,j\,+\,1}\ -\ m_{\,j}\ \geq\ n\ -\ 1$ for $1\ \leq\ j\ <\ s\,$, the factorization of $f\,$, where $A$ is the cyclotomic component, $B$ the reciprocal noncyclotomic component, $C$ the nonreciprocal part. Then, $C$ is irreducible.
\end{proposition}

\begin{proof}
Let us assume that $C$ is reducible, and apply Lemma~\ref{reducibi}. Then, there should exist $w\,(x)$ different of $\pm\,f\,(x)$ and $\pm\,f^{\,*}\,(x)$ such that $w\,(x)\cdot w^{\,*}\,(x)\ =\ f\,(x)\cdot f^{\,*}\,(x)\,$. For short, we write
\begin{equation*}
  f\,(x)\ =\ \sum_{\,j\,=\,0}^{\,r}\,a_{\,j}\,x^{\,d_{\,j}} \qquad \mbox{and} \qquad w\,(x)\ =\ \sum_{\,j\,=\,0}^{\,q}\,b_{\,j}\,x^{\,k_{\,j}}\,,
\end{equation*}
where the coefficients $a_{\,j}$ and the exponents $d_{\,j}$ are given, and the $b_{\,j}\,$'\,s and the $k_{\,j}\,$'\,s are unkown integers, with $\abs{b_{\,j}}\ \geq\ 1\,$, $0\ \leq\ j\ \leq\ q\,$,
\begin{equation*}
  a_{\,0}\ =\ -\,1\,, \quad a_{\,1}\ =\ a_{\,2}\ =\ \ldots\ =\ a_{\,r}\ =\ 1\,,
\end{equation*}
\begin{multline*}
  0\ =\ d_{\,0}\ <\ d_{\,1}\ =\ 1\ <\ d_{\,2}\ =\ n\ <\ d_{\,3}\ =\ m_{\,1}\\ 
  <\ \ldots\ <\ d_{\,r\,-\,1}\ =\ m_{\,s\,-\,1}\ <\ d_{\,r}\ =\ m_{\,s}\,,
\end{multline*}
\begin{equation*}
  0\ =\ k_{\,0}\ <\ k_{\,1}\ <\ k_{\,2}\ <\ \ldots\ <\ k_{\,q\,-\,1}\ <\ k_{\,q}\,.
\end{equation*}
The relation $w\,(x)\cdot w^{\,*}\,(x)\ =\ f\,(x)\cdot f^{\,*}\,(x)$ implies the equality: $2\,k_{\,q}\ =\ 2\,d_{\,r}\,$; expanding it and considering the terms of degree $k_{\,q}\ =\ d_{\,r}\,$, we deduce $\norm{f}^{\,2}\ =\ \norm{w}^{\,2}\ =\ r\ +\ 1$ which is equal to $s\ +\ 3\,$. Since $f^{\,*}\,(1)\ =\ f\,(1)$ and that $w^{\,*}\,(1)\ =\ w\,(1)\,$, it also implies $f\,(1)^{\,2}\ =\ w\,(1)^{\,2}$ and $b_{\,0}\cdot b_{\,q}\ =\ -\,1\,$. Then we have two equations
\begin{equation*}
  r\ -\ 1\ =\ \sum_{\,j\,=\,1}^{\,q\,-\,1}\,b_{\,j}^{\,2}\,, \qquad (r\ -\ 1)^{\,2}\ =\ \Bigl(\,\sum_{\,j\,=\,1}^{\,q\,-\,1}\,b_{\,j}\,\Bigr)^{\,2}\,.
\end{equation*}
We will show that they admit no solution except the solution $w\,(x)\ =\ \pm\,f\,(x)$ or $w\,(x)\ =\ \pm\,f^{\,*}\,(x)\,$.

Since all $\abs{b_{\,j}}\,$'\,s are $\geq\ 1\,$, the inequality $q\ \leq\ r$ necessarily holds. If $q\ =\ r\,$, then the $b_{\,j}\,$'\,s should all be equal to $-\,1$ or $1\,$, what corresponds to $\pm\,f\,(x)$ or to $\pm\,f^{\,*}\,(x)\,$. If $2\ \leq\ q\ <\ r\,$, the maximal value taken by a coefficient $b_{\,j}^{\,2}$ is equal to the largest square less than or equal to $r\ -\ q\ +\ 1\,$, so that $\abs{b_{\,j}}\ \leq\ \sqrt{r\ -\ q\ +\ 1}\,$. Therefore, there is no solution for the cases ``$q\ =\ r\ -\ 1$'' and ``$q\ =\ r\ -\ 2$''. If $q\ =\ r\ -\ 3$ all $b_{\,j}^{\,2}\,$'\,s are equal to $1$ except one equal to $4\,$, and
\begin{equation*}
  r\ -\ 1\ =\ \sum_{\,j\,=\,1}^{\,r\,-\,4}\,b_{\,j}^{\,2}\,, \qquad (r\ -\ 1)^{\,2}\ >\ \Bigl(\,\sum_{\,j\,=\,1}^{\,r\,-\,4}\,b_{\,j}\,\Bigr)^{\,2}\,.
\end{equation*}
This means that the case ``$q\ =\ r\ -\ 3$'' is impossible. The two cases ``$q\ =\ r\ -\ 4$'' and ``$q\ =\ r\ -\ 5$'' are impossible since, for $m\ =\ 5$ and $6\,$, $\sum_{\,j\,=\,1}^{\,r\,-\,m}\,b_{\,j}^{\,2}$ cannot be equal to $r\ -\ 1\,$. This is general. For $q\ \leq\ r\ -\ 3$ at least one of the $\abs{b_{\,j}}\,$'\,s is equal to $2\,$; in this case we would have 
\begin{equation*}
  r\ -\ 1\ =\ \pm\,\sum_{\,j\,=\,1}^{\,q\,-\,1}\,b_{\,j}\ \leq\ \sum_{\,j\,=\,1}^{\,q\,-\,1}\,\abs{b_{\,j}}\ <\ \sum_{\,j\,=\,1}^{\,q\,-\,1}\,b_{\,j}^{\,2}\ =\ r\ -\ 1\,.
\end{equation*}
Contradiction.
\end{proof}


\subsection[Thickness of the annular neighbourhoods of $\abs{z}\ =\ 1$]{Thickness of the annular neighbourhoods of $\abs{z}\ =\ 1$ containing the nonlenticular roots}
\label{S5.4}

Let $n\ \geq\ 3\,$, and $\delta_{\,n}$ be a real number $>\ 0\,$, smaller than $1\,$. Let
\begin{equation*}
  \D_{\,n,\,\delta_{\,n}}\ :=\ \bigl\{\,z\ \bigl\vert\ \abs{z}\ <\ 1\,,\ \delta_{\,n}\ <\ \abs{G_{\,n}\,(z)}\,\}\,.
\end{equation*}
We now characterize the geometry of the zeroes, in $\D_{\,n,\,\delta_{\,n}}\,$, of a given
\begin{equation*}
  f\,(x)\ :=\ -\,1\ +\ x\ +\ x^{\,n}\ +\ x^{\,m_{\,1}}\ +\ x^{\,m_{\,2}}\ +\ \ldots\ +\ x^{\,m_{\,s}}\ \in\ \B_{\,n}\,,
\end{equation*}
where $s\ >\ 0\,$, $m_{\,1}\ -\ n\ \geq\ n\ -\ 1\,$, $m_{\,q\,+\,1}\ -\ m_{\,q}\ \geq\ n\ -\ 1$ for $1\ \leq\ q\ <\ s\,$. Obviously a zero $x\ \in\ \D_{\,n,\,\delta_{\,n}}$ of $f$ is $\neq\ 0$ and is not a zero of the trinomial $-\,1\ +\ x\ +\ x^{\,n}\,$. Moreover,
\begin{multline}\label{deltamino}
  \delta_{\,n}\ <\ \abs{-\,1\ +\ x\ +\ x^{\,n}}\ =\ \abs{x^{\,m_{\,1}}\ +\ x^{\,m_{\,2}}\ +\ \ldots\ +\ x^{\,m_{\,s}}}\\ 
  <\ \abs{x}^{\,m_{\,1}}\ +\ \abs{x}^{\,m_{\,2}}\ +\ \ldots\ +\ \abs{x}^{\,m_{\,s}}\,.
\end{multline}
This inequality implies that $1\ -\ \abs{x}$ is necessarily small. Indeed the function $Y\,:\ u\ \longmapsto\ \sum_{\,j\,=\,1}^{\,s}\,u^{\,m_{\,j}}$ is increasing, with increasing derivative, on $(\,0,\,1\,]\,$, so that the unique real value $0\ <\ r\ <\ 1$ which satisfies $Y\,(r)\ =\ \delta_{\,n}$ admits the upper bound $e_{\,\sup}\ <\ 1$ given by $s\ -\ \delta_{\,n}\ =\ Y^{\,\prime}\,(1)\cdot(1\ -\ e_{\,\sup})\ =\ \Bigl(\,\sum_{\,j\,=\,1}^{\,s}\,m_{\,j}\,\Bigr)\cdot(1\ -\ e_{\,\sup})\,$; so that
\begin{equation*}
  r\ <\ e_{\,\sup}\ =\ 1\ -\ \frac{s\ -\ \delta_{\,n}}{\sum_{\,j\,=\,1}^{\,s}\,m_{\,j}}\,.
\end{equation*}
Let us now give a lower bound $e_{\,\inf}$ of $r\,$, as a function of $n\,$, $s\,$, $\delta_{\,n}$ and $m_{\,s}\,$. If $s\ =\ 1\,$, using $m_{\,1}\ \geq\ n\ +\ (n\ -\ 1)\,$, the inequality $\delta_{\,n}\ \leq\ \abs{x}^{\,m_{\,1}}\ \leq\ \abs{x}^{\,2\,n\,-\,1}$ implies:
\begin{equation*}
  e_{\,\inf}\ =\ \delta_{\,n}^{\,1/(2\,n\,-\,1)}\ \leq\ r\,.
\end{equation*}
As soon as the assumption $\limsup\limits_{n\ \to\ \infty}\,\dfrac{\lo\delta_{\,n}}{n}\ =\ 0$ is satisfied, then $e_{\,\inf}$ tends to $1$ as $n$ tends to infinity. This assumption means that the domain $\D_{\,n,\,\delta_{\,n}}$ should avoid small disks centered at the lenticular roots of $G_{\,n}\,$.

If $s\ \geq\ 2\,$, using the inequalities $m_{\,q\,+\,1}\ -\ m_{\,q}\ \geq\ n\ -\ 1\,$, $1\ \leq\ q\ <\ s\,$, we deduce, from Equation~\eqref{deltamino},
\begin{multline*}
  \delta_{\,n}\ <\ \abs{x}^{\,m_{\,1}}\ +\ \abs{x}^{\,m_{\,2}}\ +\ \ldots\ +\ \abs{x}^{\,m_{\,s}}\\ 
  \leq\ \abs{x}^{\,2\,n\,-\,1}\,\Bigl(\frac{1\ -\ \abs{x}^{\,(n\,-\,1)\cdot(s\,-\,1)}}{1\ -\ \abs{x}^{\,n\,-\,1}}\Bigr)\ +\ \abs{x}^{\,m_{\,s}}\,.
\end{multline*}
Putting $H\ =\ \abs{x}^{\,n\,-\,1}\,$, we are now bound to solve the following equation in $H$
\begin{equation*}
  \delta_{\,n}\ =\ H^{\,2}\cdot\Bigl(\,\frac{1\ -\ H^{\,s\,-\,1}}{1\ -\ H}\,\Bigr)\ +\ H^{\,m_{\,s}\,-\,1}
\end{equation*}
to find $e_{\,\inf}\,$, for $H\ <\ 1$ close to one, of the form $1\ -\ \epsilon\,$. It is easy to check that the expression of $\epsilon\,$, at the first-order, is
\begin{equation*}
  \epsilon\ =\ 2\;\frac{n\,s\ -\ \delta_{\,n}\,n\ +\ \delta_{\,n}\ -\ s}{n\,s^{\,2}\ +\ n\,s\ -\ s^{\,2}\ +\ 2\,m_{\,s}\ -\ 2\,n\ -\ s}\,,
\end{equation*}
leading to
\begin{equation*}
  e_{\,\inf}\ =\ \left(1\ -\ 2\;\frac{n\,s\ -\ \delta_{\,n}\,n\ +\ \delta_{\,n}\ -\ s}{n\,s^{\,2}\ +\ n\,s\ -\ s^{\,2}\ +\ 2\,m_{\,s}\ -\ 2\,n\ -\ s}\right)^{\,1\,/\,(n\,-\,1)}\,.
\end{equation*}
For $n\,$, $\delta_{\,n}$ and $s$ fixed, the function $m_{\,s}\ \longmapsto\ \epsilon$ is decreasing and then $m_{\,s}\ \longmapsto\ e_{\,\inf}$ is increasing. This means that the thickness of the annular neighbourhood of $\abs{z}\ =\ 1$ containing the nonlenticular roots of $f$ diminishes as the degree $m_{\,s}$ of $f$ tends to infinity, for a fixed number of monomials $s\ +\ 3$ and a fixed dynamical degree $n\,$.

Therefore all the zeroes of $f$ which lie in $\D_{\,n,\,\delta_{\,n}}$ belong to
\begin{equation*}
  \bigl\{\,z\ \bigl\vert\ e_{\,\inf}\ <\ z\ <\ 1\,\bigr\}\,.
\end{equation*}
An example of dependency of $e_{\,\inf}$ with $m_{\,s}$ is given by Figure~\ref{lt05}(\textit{b}) and Figure~\ref{lt02}(\textit{b}): for fixed $n\ =\ 12$ and $s\ =\ 5\,$, and varying $m_{\,s}$ from $35$ to $385\,$.

The \textsc{Monte--Carlo} approach allows to compare the thicknesses $\delta_{\,n}$ with numerical values of thicknesses $\delta\,(n)$ obtained by numerical computation of roots for polynomials of $\B$ with $n\ \leq\ 3\,000\,$. The results are reported in Figure~\ref{thicknesses}.

\begin{figure}
  \begin{center} 
    \includegraphics[width=0.89\textwidth]{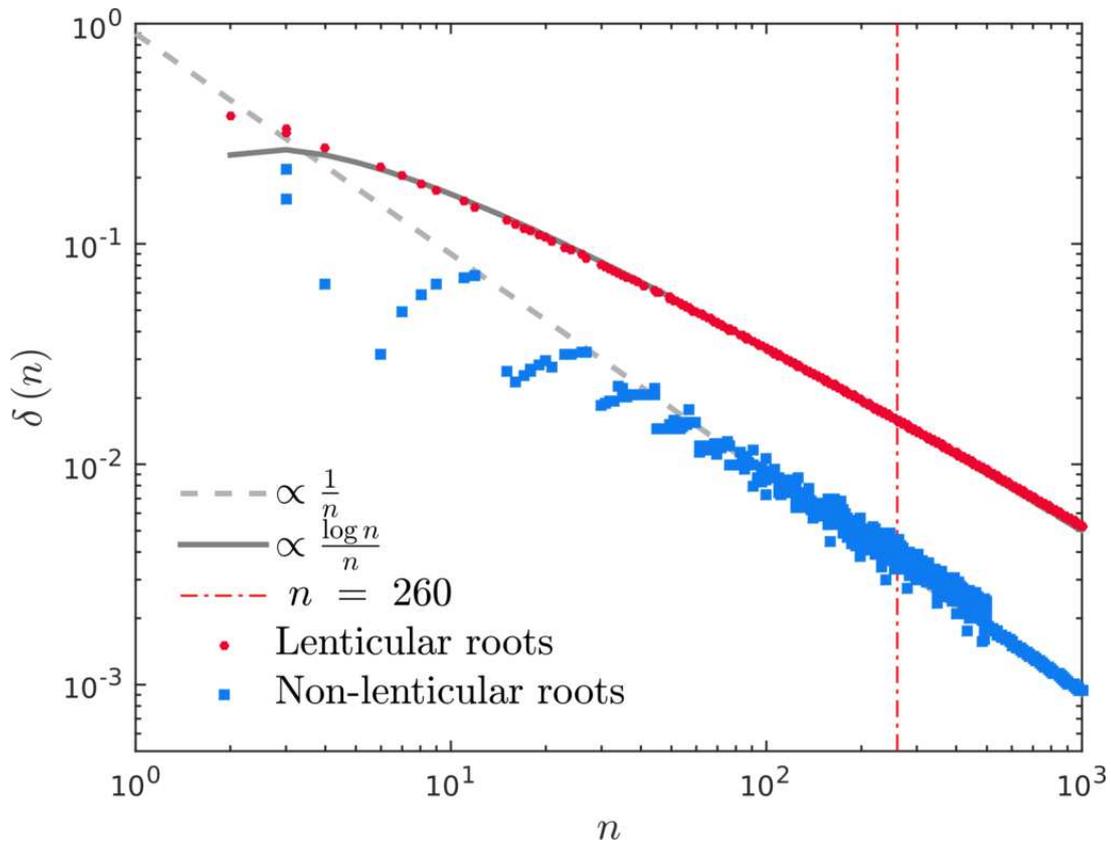}
  \end{center}
  \caption{\small\em Thickness $\delta\,(n)\,$, proportional to $1/n$ at the first-order, of the annular neighbourhood of the unit circle which contains the nonlenticular roots (of modulus $<\ 1$), represented as a function of the dynamical degree $n\,$, for the almost \textsc{Newman} polynomials $f$ of the class $\B\,$. The curve ``lenticular roots'' represents the distance between $1$ and the positive real zero (summit) of the lenticulus of $f\ \in\ \B\,$; this distance is $\sim\ \lo n/n$ at the first-order. The \textsc{Monte-Carlo} method in $\B_{\,n}$ is used, for $n$ less than $3\,000\,$.}
  \label{thicknesses}
\end{figure}


\subsection{Proof of Theorem~\ref{thm1factorization}}
\label{S5.5}

\begin{enumerate}
  \item By Proposition~\ref{descartesrule} the nonreciprocal part $C$ is nontrivial. By Proposition~\ref{nonreciprocalpart} the nonreciprocal part $C$ is irreducible. By Proposition~\ref{neverunitcircle} the irreducible factor $C$ never vanishes on the unit circle.

  \item For $n\ \geq\ 3$ the \textsc{R\'enyi} $\beta-$expansion of $1$ in base $\theta_{\,n}^{\,-1}\ >\ 1$ is the sequence of digits of the coefficient vector of $G_{\,n}\,(x)\ +\ 1$ (\cf \textsc{Lothaire} \cite[Chap.~7]{Lothaire2002}); the digits lie in the alphabet $\{0,\, 1\}\,$. We have
  \begin{equation*}
    d_{\,\theta_{\,n}^{\,-\,1}}\,(1)\ =\ 0\,.\,1\,0^{\,n\,-\,2}\,1\,.
  \end{equation*}
  Similarly $d_{\,\theta_{\,n\,-\,1}^{\,-1}}\,(1)\ =\ 0\,.\,1\,0^{\,n\,-\,3}\,1\,$, where the sequence of digits comes from the coefficient vector of $G_{\,n\,-\,1}\,(x)\ +\ 1\,$. Let $\beta\ >\ 1$ denote the real algebraic integer such that the \textsc{R\'enyi} $\beta-$expansion of $1$ in base $\beta$ is exactly the sequence of digits of the coefficient vector of $f\,(x)\ +\ 1\,$. We have:
  \begin{equation*}
    d_{\,\beta}\,(1)\ =\ 0\,.\,1\,0^{\,n\,-\,2}\,1\,0^{\,m_{\,1}\,-\,n\,-\,1}\,1\,0^{\,m_{\,2}\,-\,m_{\,1}\,-\,1}\,1\,\ldots\,1\,0^{\,m_{\,s}\,-\,m_{\,s\,-\,1}\,-\,1}\,1\,.
  \end{equation*}
  Since the two following lexicographical conditions are satisfied:
  \begin{equation*}
    d_{\,\theta_{\,n}^{\,-1}}\,(1)\ =\ 0\,.\,1\,0^{\,n\,-\,2}\,1\ \preccurlyeq_{lex}\ d_{\,\beta}\,(1)\ \preccurlyeq_{lex}\ d_{\,\theta_{\,n\,-\,1}^{\,-1}}\,(1)\ =\ 0\,.\,1\,0^{\,n\,-\,3}\,1\,.
  \end{equation*}
  Lemma~5.1 (ii) in \textsc{Flatto, Lagarias} and \textsc{Poonen} \cite{Flatto1994} implies:
  \begin{equation*}
    \theta_{\,n}^{\,-1}\ <\ \beta\ <\ \theta_{\,n\,-\,1}^{\,-1} \quad\Longleftrightarrow\quad \theta_{\,n\,-\,1}\ <\ 1/\beta\ <\ \theta_{\,n}\,.
  \end{equation*}
  Since $-\,C^{\,*}$ is nontrivial, monic, irreducible, nonreciprocal, and vanishes at $\beta\,$, it is the minimal polynomial of $\beta\,$, and $\beta$ is nonreciprocal.
\end{enumerate}


\section{Heuristics on the irreducibility of the polynomials of $\B$}
\label{S6}

The \textsc{Monte-Carlo} method is used for testing the \textsc{Odlyzko--Poonen} Conjecture (``OP Conjecture'') on the \textsc{Newman} polynomials, the variant Conjecture (``variant OP Conjecture'') on the almost \textsc{Newman} polynomials and for estimating the proportion of irreducible polynomials in the class $\B\,$. The Conjectures ``OP'' and ``variant OP'' state that the proportion of irreducible polynomials in the class of \textsc{Newman} polynomials, resp. almost \textsc{Newman} polynomials, is one. This value of one is reasonable in the context of the general Conjectures on random polynomials \cite{Borst2017}.

The probability of $f\ \in\ \B$ to be an irreducible polynomial can be defined asymptotically as follows. Let $s\ \geq\ 1\,$, $n\ \geq\ 2$ and $N\ \geq\ n\,$. Let $\B_{\,n}^{\,(N,\,s)}$ denote the set of the polynomials $f\ \in\ \B_{\,n}$  having $s\ +\ 3$ monomials such that $\deg\,(f)\ \leq\ N\,$. Denote 
\begin{equation*}
  \B^{\,(N,\,s)}\ :=\ \bigcup_{\,2\ \leq\ n\ \leq\ N}\,\B_{\,n}^{\,(N,\,s)}\,, \qquad \B^{\,(N)}\ :=\ \bigcup_{\,s\ \geq\ 1}\,\B^{\,(N,\,s)}\,.
\end{equation*}
Then $\B\ =\ \bigcup_{\,N\ \geq\ 2}\,\B^{\,(N)}\,$. For $s\ \geq\ 1\,$, let $\B^{\,[s]}\ =\ \bigcup_{\,N\ \geq\ 2}\,\B^{\,(N,\,s)}\,$. For every $s\ \geq\ 1$ though the two adherence values
\begin{multline}\label{eqq1}
  \liminf_{\,N\ \to\ \infty}\;\frac{\#\{\,f\ \in\ \B^{\,(N,\,s)}\ \bigl\vert\ f\ \mbox{irreducible}\,\bigr\}}{\#\bigl\{\,f\ \in\ \B^{\,(N,\,s)}\,\bigr\}}\ \leq\\ 
  \limsup_{\,N\ \to\ \infty}\;\frac{\#\bigl\{\,f\ \in\ \B^{\,(N,\,s)}\ \bigl\vert\ f\ \mbox{irreducible}\,\bigr\}}{\#\bigl\{\,f\ \in\ \B^{\,(N,\,s)}\,\bigr\}}\,,
\end{multline}
exist, and, in a similar way,
\begin{multline}\label{eqq2}
  \liminf_{\,N\ \to\ \infty}\;\frac{\#\bigl\{\,f\ \in\ \B^{\,(N)}\ \bigl\vert\ f\ \mbox{irreducible}\,\bigr\}}{\#\bigl\{\,f\ \in\ \B^{\,(N)}\,\bigr\}}\ \leq\\ 
  \limsup_{\,N\ \to\ \infty}\frac{\#\bigl\{\,f\ \in\ \B^{\,(N)}\ \bigl\vert\ f\ \mbox{irreducible}\,\bigr\}}{\#\bigl\{\,f\ \in\ \B^{\,(N)}\,\bigr\}}\,,
\end{multline}
exist, without being equal a priori, we find that, for $s\ =\ 1$ and $s\ =\ 2\,$, and for arbitrary values of $s\ \geq\ 1\,$, there is a numerical evidence that the limits exist in both \eqref{eqq1} and \eqref{eqq2} (\ie $\liminf\ =\ \limsup$). Table~\ref{tableIRRED} reports the proportion of irreducible quadrinomials ($s\ =\ 1$), resp. irreducible pentanomials ($s\ =\ 2$), in the class $\B\,$, with the $90\%-$confidence interval under the assumption that the  limit exists in each case. We find that the proportion of irreducible polynomials in $\B$ is 
\begin{equation*}
  \lim_{\,N\ \to\ \infty}\;\frac{\#\bigl\{\,f\ \in\ \B^{\,(N)}\ \bigl\vert\ f\ \mbox{irreducible}\,\bigr\}}{\#\bigl\{\,f\ \in\ \B^{\,(N)}\,\bigr\}}\ =\ 0.756\ \pm\ 0.02235\,.
\end{equation*}
This value justifies the statement of the ``\emph{Asymptotic Reducibility Conjecture}''. The reason of this residual reducibility finds its origin in Proposition~\ref{beaucoup} where cyclotomic polynomials are asymptotically present in the factorizations, though the authors have no proof of it. By \textsc{Monte-Carlo} methods, polynomials of degrees $N$ up to $3\,000$ are tested (see Figure~\ref{IrreducibilityMontecarlo}), and the number of monomials $s\ +\ 3$ in each $f\ \in\ \B_{\,n}^{\,(N)}$ is random in the range of values of $s\,$.

\begin{table}
  \caption{\small\em Asymptotic proportion of irreducible polynomials in various classes: \textsc{Newman} polynomials, almost \textsc{Newman} polynomials, $\B$ and the subclasses $\B^{\,[\,0\,]}\,$, $\B^{\,[\,1\,]}\,$, $\B^{\,[\,2\,]}$ of $\}$ (Maximal polynomial degree: $3\,000\,$, number of \textsc{Monte-Carlo} runs: $4\,000$)}
  \bigskip
  \begin{tabular}{lccl}
  \hline\hline\noalign{\smallskip}
  Polynomials (class) & Proportion & $90\%-$confidence & Expected\\ & & Interval (estimated) & \\
  \noalign{\smallskip}\hline\hline\noalign{\smallskip}
  OP (\textsc{Newman}) & $0.967$ & $0.00930$ & $1$ (Conjectured) \\
  variant OP (almost \textsc{Newman}) & $0.968$ & $0.00916$ & $1$ (Conjectured) \\
  Class $\B$ & $0.756$ & $0.02235$ & $3/4$ (Conjectured) \\
  Trinomials ($s\ =\ 0$) & $5/6\ =\ 0.8(3)$ & --- & $5/6$ ~exact (\textsc{Selmer}) \\
  Quadrinomials ($s\ =\ 1$)& $0.575$ & $0.02573$ & unkown \\
  Pentanomials ($s\ =\ 2$) & $0.826$ & $0.01601$ & unkwon \\
  \noalign{\smallskip}\hline\hline
  \end{tabular}
  \label{tableIRRED}
\end{table}

In the case ``$s\ =\ 0$'', $\bigcup_{\,n\ \geq\ 2}\,\B^{\,(N\,=\,n,\,s\,=\,0)}$ denotes the set of trinomials of the type $-\,1\ +\ x\ +\ x^{\,n}\,$, $n\ \geq\ 2\,$, whose factorization was studied by \textsc{Selmer} \cite{Selmer1956}; the proportion of irreducible trinomials is exact:
\begin{equation}\label{eqq3}
  \lim_{\,n\ \to\ \infty}\;\frac{\#\bigl\{\,f\ \in\ \B^{\,(n,\,s\,=\,0)}\ \bigl\vert\ f\ \mbox{irreducible}\,\bigr\}}{\#\bigl\{\,f\ \in\ \B^{\,(n,\,s\,=\,0)}\}}\ =\ \frac{5}{6}\ =\ 0.8(3)\,.
\end{equation}

\begin{figure}
  \begin{center}
    \includegraphics[width=0.89\textwidth]{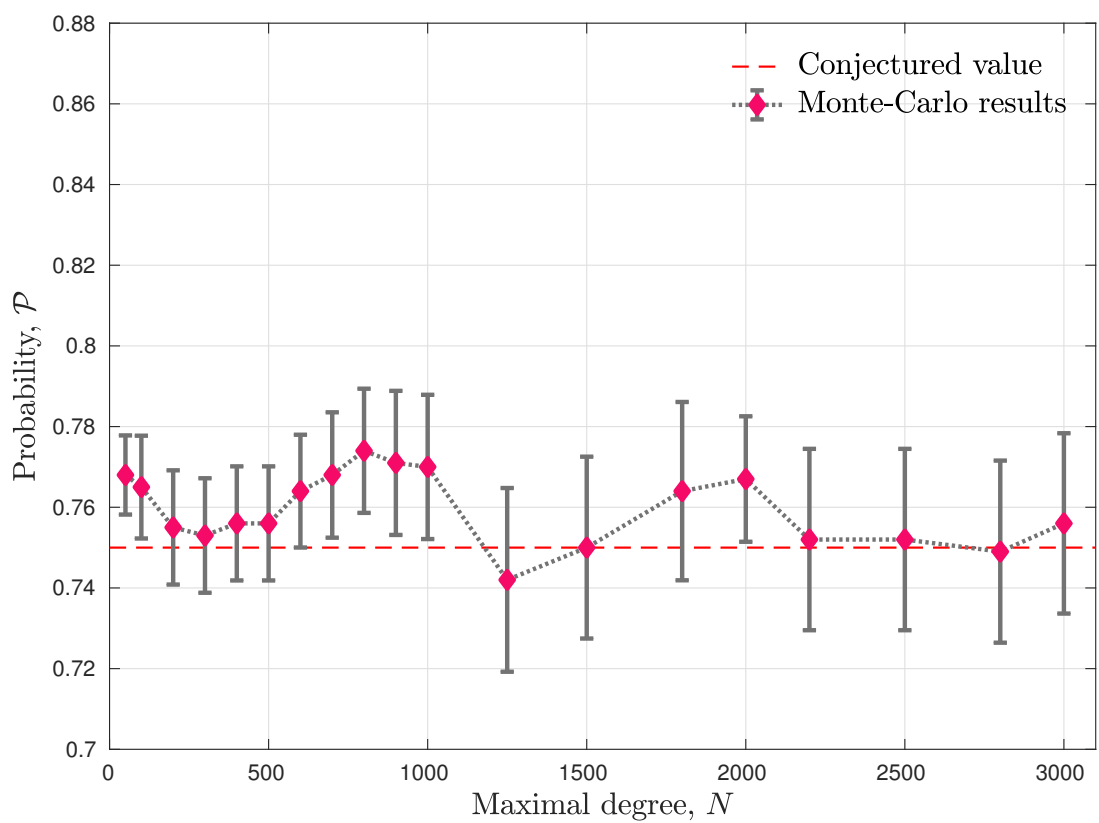}
  \end{center}
  \caption{\small\em Probability to be irreducible for a polynomial of the class $\B$ having degree less than $N\,$. The estimated $90\%-$confidence intervals are represented. A limit value, as $N$ tends to infinity, is conjectured to exist and its value is conjectured to be the rational number $3/4\,$.}
  \label{IrreducibilityMontecarlo}
\end{figure}


\section[Lenticular roots on continuous curves stemming from $z\ =\ 1$]{Lenticular roots on continuous curves stemming from $z\ =\ 1$ and boundary of Solomyak's fractal}
\label{S7}

In this paragraph we first recall the constructions of \textsc{Solomyak} \cite{Solomyak1994} on the sets of zeroes of the family $\W$ of power series having real coefficients in the interval $[\,0,\,1\,]\,$, in the interior of the unit disk, and \textsc{Solomyak}'\,s Theorem~\ref{solomyakfractalGBordG}. Then, we will recall how the polynomials of the class $\B$ are related to elements of $\W\,$.

Let
\begin{equation*}
  \W\ :=\ \bigl\{\,h\,(z)\ =\ 1\ +\ \sum_{\,j\,=\,1}^{\,+\,\infty}\,a_{\,j}\,z^{\,j}\ \bigl\vert\ a_{\,j}\ \in\ [\,0,\,1\,]\,\bigr\}
\end{equation*}
be the class of power series defined on $\abs{z}\ <\ 1$ equipped with the topology of uniform convergence on compacts sets of $\abs{z}\ <\ 1\,$. The subclass $\W_{\;0,\,1}$ of $\W$ denotes functions whose coefficients are all zeros or ones. The space $\W$ is compact and convex. Let
\begin{equation*}
  \G\ :=\ \bigl\{\,\lambda\ \bigl\vert\ \abs{\lambda}\ <\ 1\,,\ \exists\,h\,(z)\ \in\ \B\ \mbox{such that}\ h\,(\lambda)\ =\ 0\,\bigr\}\ \subset\ \bigl\{\,z\ \bigl\vert\ \abs{z}\ <\ 1\,\bigr\}
\end{equation*}
be the set of zeroes of the power series belonging to $\W\,$. The elements of $\G$ lie within the unit circle and curves in $\abs{z}\ <\ 1$ given in polar coordinates, close to the unit circle, by \cite{Verger-Gaugry2016}. The domain $D\,(0,\,1)\ \setminus\ \G$ is star-convex due to the fact that: $h\,(z)\ \in\ \W\quad \Longrightarrow\quad h\,(z\,/\,r)\ \in\ \W\,$, for any $r\ >\ 1$ (\cf \cite[Section~\textsection3]{Solomyak1994}).

For every $\phi\ \in\ (0,\,2\,\pi)\,$, there exists $\lambda\ =\ r\,\ue^{\,\ui\,\phi}\ \in\ \G\,$; the point of minimal modulus with argument $\phi$ is denoted $\lambda_{\,\phi}\ =\ \rho_{\,\phi}\,\ue^{\,\ui\,\phi}\ \in\ \G\,$, $\rho_{\,\phi}\ <\ 1\,$. A function $h\ \in\ \W$ is called $\phi-$optimal if $h\,(\lambda_{\,\phi})\ =\ 0\,$. Denote by $\K$ the subset of $(0,\,\pi)$ for which there exists a $\phi-$optimal function belonging to $\W_{\;0,\,1}\,$. Denote by $\partial\,\G_{\,S}$ the ``spike'': $\,\Bigl[-\,1,\,\frac{1}{2}\;(1\ -\ \sqrt{5})\,\Bigr]$ on the negative real axis.

\begin{theorem}[Solomyak]\label{solomyakfractalGBordG}
\begin{enumerate}

  \item The union $\G\ \bigcup\ \T\ \bigcup\ \partial\,\G_{\,S}$ is closed, symmetrical with respect to the real axis, has a cusp at $z\ =\ 1$ with logarithmic tangency (\cf \cite[Figure~1]{Solomyak1994}).

  \item The boundary $\partial\,\G$ is a continuous curve, given by $\phi\ \longmapsto\ \abs{\lambda_{\,\phi}}$ on $[\,0,\,\pi)\,$, taking its values in $\Bigl[\,\frac{\sqrt{5}\ -\ 1}{2},\, 1\Bigr)\,$, with $\abs{\lambda_{\,\phi}}\ =\ 1$ if and only if $\phi\ =\ 0\,$. It admits a left-limit at $\pi^{\,-}\,$, $1\ >\ \lim_{\,\phi\ \to\ \pi^{\,-}}\,\abs{\lambda_{\,\phi}}\ >\ \abs{\lambda_{\,\pi}}\ =\ \frac{1}{2}\;(-\,1\ +\ \sqrt{5})\,$, the left-discontinuity at $\pi$ corresponding to the extremity of $\partial\,\G_{\,S}\,$.

  \item At all points $\rho_{\,\phi}\,\ue^{\,\ui\,\phi}\ \in\ \G$ such that $\phi\,/\,\pi$ is rational in an open dense subset of $(0,\,2)\,$, $\partial\,\G$ is non-smooth.

  \item There exists a nonempty subset of transcendental numbers $L_{\,\mathrm{tr}}\,$, of \textsc{Hausdorff} dimension zero, such that $\phi\ \in\ (0,\,\pi)$ and $\phi\ \not\in\ \K\ \bigcup\ \pi\ \Q\ \cup\ \pi\,L_{\,\mathrm{tr}}$ implies that the boundary curve $\partial\,\G$ has a tangent at $\rho_{\,\phi}\,\ue^{\,\ui\,\phi}$ (smooth point).

\end{enumerate}
\end{theorem}

\begin{proof}
\cite[Sections~\textsection3 and \textsection4]{Solomyak1994}.
\end{proof}

\begin{figure}
  \begin{center}
    \subfigure[]{\includegraphics[width=0.28\textwidth]{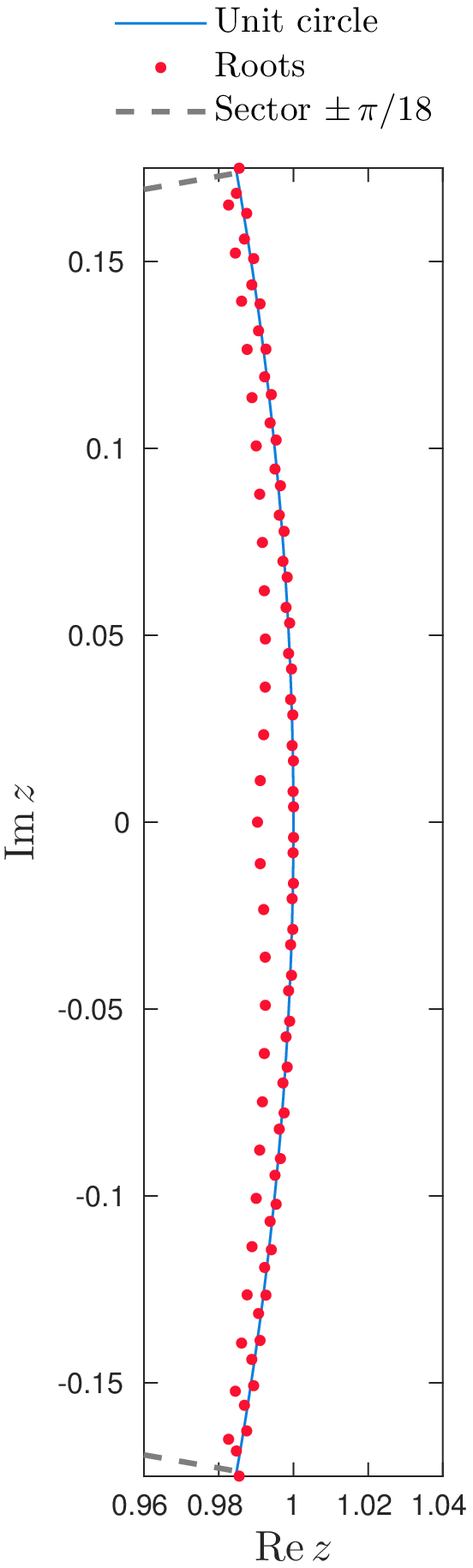}}
    \subfigure[]{\includegraphics[width=0.71\textwidth]{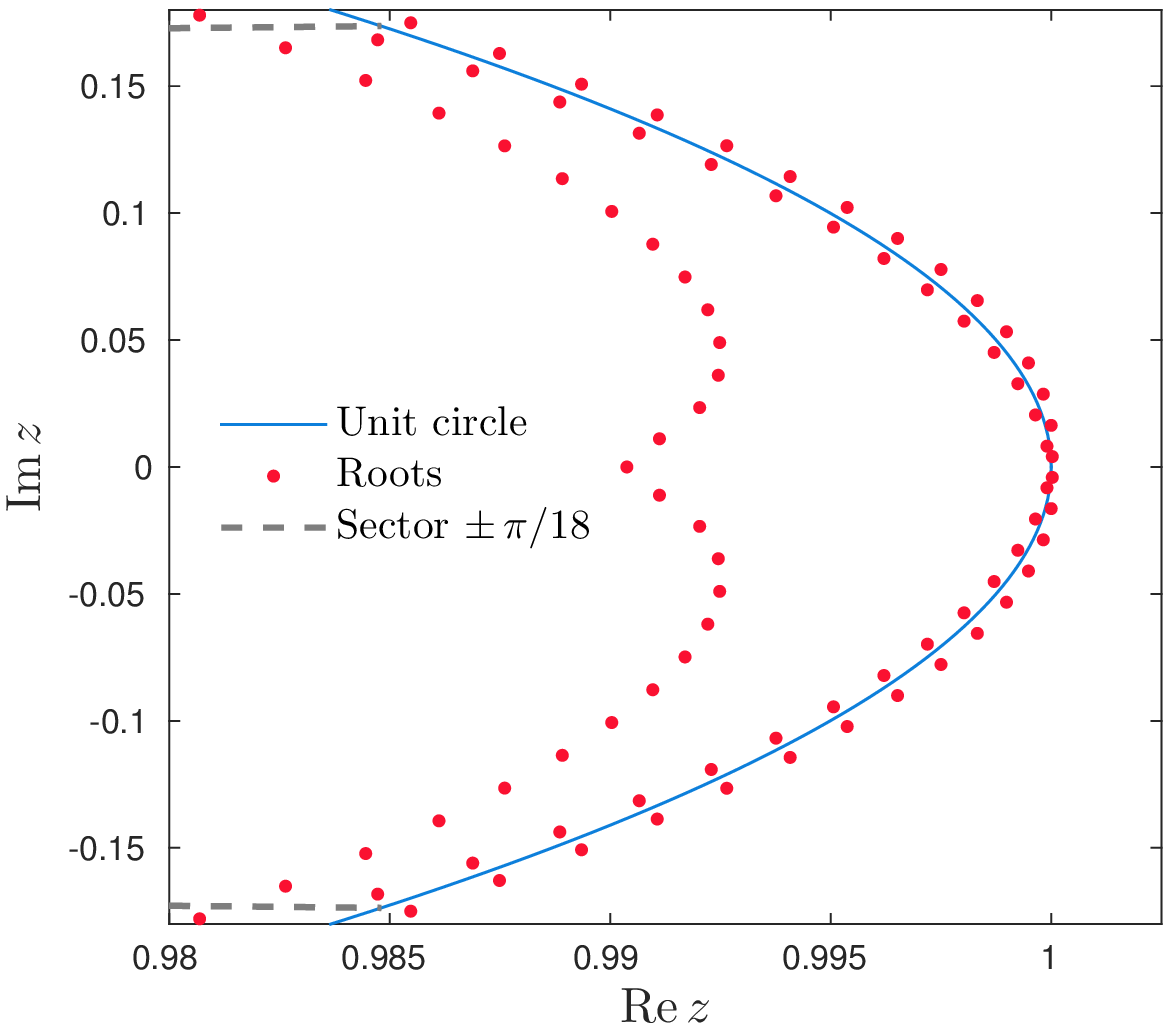}}
  \end{center}
  \caption{\small\em The representation of the $27$ zeroes of the lenticulus of $f\,(x)\ =\ -\,1\ +\ x\ +\ x^{\,481}\ +\ x^{\,985}\ +\ x^{\,1502}$ in the angular sector $-\,\pi/18\ <\ \arg\,z\ <\ \pi/18$ in two different scalings in $x$ and $y$ (in (a) and (b)). In this angular sector the other zeroes of $f\,(x)$ can be found in a thin annular neighbourhood of the unit circle. The real root $1/\beta\ >\ 0$ of $f\,(x)$ is such that $\beta$ satisfies: $1.00970357\ldots\ =\ \theta_{\,481}^{\,-1}\ <\ \beta\ =\ 1.0097168\ldots\ <\ \theta_{\,480}^{\,-1}\ =\ 1.0097202\ldots\,$.}
  \label{ltlent481}
\end{figure}

\begin{figure}
  \begin{center}
    \includegraphics[width=0.95\textwidth]{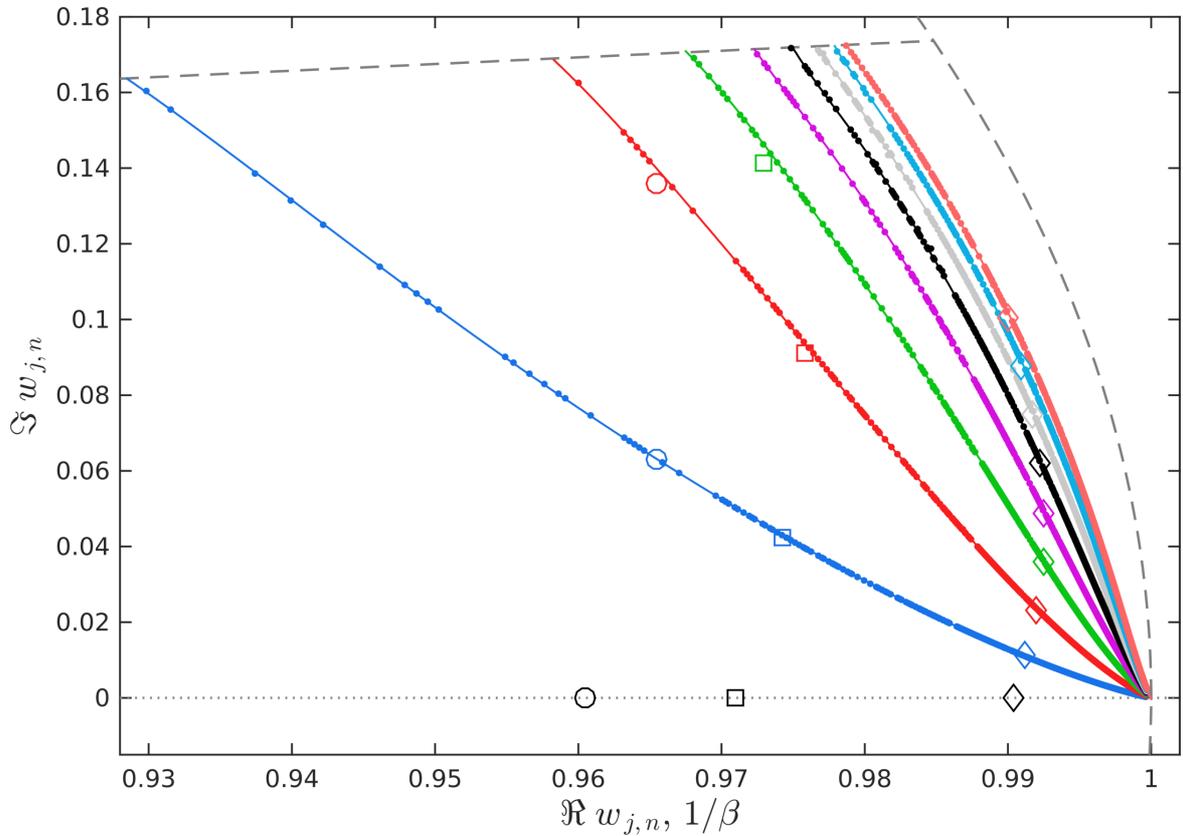}
  \end{center}
  \caption{\small\em Curves stemming from $1$ which constitute the lenticular zero locus of all the polynomials of the class $\B\,$. These (universal) curves are continuous. The first one above the real axis, corresponding to the zero locus of the first lenticular roots, lies inside the boundary of \textsc{Solomyak}'\,s fractal \cite{Solomyak1994}. The lenticular roots of the polynomials $f$ in the examples of the Figures~\ref{lt06}, \ref{lt07} and \ref{ltlent481} are represented by the respective symbols $\circ\,$, $\square\,$, $\diamond\,$. The dashed lines represent the unit circle and the top boundary of the angular sector $\abs{\arg\,z}\ <\ \pi/18\,$. The complete set of curves, \ie the locus of lenticuli, is obtained by symmetrization with respect to the real axis.}
  \label{solomyakEXTENDEDcurves}
\end{figure}

Let $\beta\ >\ 1$ be a real number and $T_{\,\beta}\,:\ [\,0,\,1\,]\ \longmapsto\ [\,0,\,1\,]\,$, $x\ \longmapsto\ \beta\,x\ -\ \lfloor\,\beta\,x\,\rfloor\ =\ \{\,\beta\,x\,\}$ be the $\beta-$transformation. The $i-$iterate of $T_{\,\beta}$ is denoted by $T_{\,\beta}^{\,i}\,$. The orbit $(T_{\,\beta}^{\,i}\,(1))_{\,i\ \geq\ 1}$ of $1$ in the interval $[\,0,\,1\,]$ defines the sequence $(t_{\,i})$ of digits $t_{\,i}\ :=\ \lfloor\,\beta\,T_{\,\beta}^{\,i\,-\,1}\,(1)\,\rfloor\,$, which belong to the alphabet $\{\,0,\,1\,\}$ and satisfy the conditions of \textsc{Parry} (\textsc{Lothaire} \cite[Chap.~7]{Lothaire2002}). The \textsc{Parry} Upper function $f_{\,\beta}\,(z)$ at $\beta$ is defined as the power series having coefficient vector: ``$-1\ t_{\,1}\ t_{\,2}\ t_{\,3}\ \ldots\,$''. When the \textsc{Parry} Upper function $f_{\,\beta}\,(z)$ at $\beta$ is a polynomial, by Lemma~5.1 (ii) in \cite{Flatto1994}, and
\begin{equation*}
  1\ <\ \beta\ <\ \theta_{\,2}^{\,-1}\ =\ \frac{1\ +\ \sqrt{5}}{2}\,,
\end{equation*}
the Conditions of \textsc{Parry} are exactly expressed by the defining conditions
\begin{equation*}
  n\ \geq\ 3\,, \quad s\ \geq\ 0\,, \quad m_{\,1}\ -\ n\ \geq\ n\ -\ 1\,, \quad m_{\,q\,+\,1}\ -\ m_{\,q}\ \geq\ n\ -\ 1 \quad \mbox{for} \quad 1\ \leq\ q\ <\ s
\end{equation*}
of the polynomial $f$ of the class $\B$ in \eqref{basic}, with $f\,(x)\ =\ -\,1\ +\ t_{\,1}\,x\ +\ t_{\,2}\,x^{\,2}\ +\ t_{\,3}\,x^{\,3}\ +\ \ldots\,$. The polynomials $f$ of the class $\B$ can be viewed as all the polynomial sections of all the \textsc{Parry} Upper functions $f_{\,\beta}\,(z)$ at $\beta$ for all
$1\ <\ \beta\ <\ \theta_{\,2}^{\,-1}\,$. The correspondance $\beta\ \longmapsto\ f_{\,\beta}\,(z)$ is one-to-one \cite{Verger-Gaugry2008}.

Now the identity $\lfloor\,\beta\,T_{\,\beta}^{\,i\,-\,1}\,(1)\,\rfloor\ =\ 
\beta\,T_{\,\beta}^{\,i\,-\,1}\,(1)\ -\ T_{\,\beta}^{\,i}\,(1)\,$, $i\ \geq\ 2\,$, implies the factorization
\begin{equation*}
  -\,1\ +\ t_{\,1}\,x\ +\ t_{\,2}\,x^{\,2}\ +\ t_{\,3}\,x^{\,3}\ +\ \ldots\ =\ -\,(1\ -\ \beta\,x)\cdot\Bigl(1\ +\ \sum_{\,j\,\geq\,1}\,T_{\,\beta}^{\,i}\,(1)\,x^{\,i}\Bigr)
\end{equation*}
for which the second factor belongs to $\W\,$. Hence, except the collection of the real zeroes $1/\beta$ which are those of the polynomials $f\ \in\ \B$ in $[\,0,\,1\,]\,$, all the zeroes of the polynomials $f\ \in\ \B\,$, of modulus $<\ 1\,$, lie within \textsc{Solomyak}'\,s fractal domain $\G\,$, having boundary described by Theorem~\ref{solomyakfractalGBordG}. By construction the zero locus of the first roots in Figure~\ref{solomyakEXTENDEDcurves} is included in this boundary. Therefore it has logarithmic tangency at $z\ =\ 1\,$. The zero loci of the second roots, third roots, \etc, closer to $\abs{z}\ =\ 1\,$, in Figure~\ref{solomyakEXTENDEDcurves}, lie within $\G\,$. In Figure~\ref{solomyakEXTENDEDcurves} are represented these (universal) curves on which the zeroes of the preceding examples are reported. A complete study of these curves will be performed in the nearest future.


\section*{Acknowledgements}

The authors would like to thank Dr.~Bill~\textsc{Allombert} (Institut de Math\'ematiques de Bordeaux, France) for helpful discussions on PARI/GP software.


\appendix
\section{Algorithms and programs}

The pseudo-code of the employed \textsc{Monte-Carlo} algorithm and the PARI/GP program~\ref{list} used in the present study is given below:
\begin{algorithm}[H]
  \bigskip
  \begin{algorithmic}
    \Require $N_{\,\max}\ \in\ \N$ \Comment{\it Maximal polynomial degree}
    \Require $M\ \in\ \N$ \Comment{\it Number of Monte-Carlo drawings}
    \State Irreducible[1:$M$] $\gets\ 0$
    \For{$k\ =\ 1$ to $M$}
      \State $N\ \gets\ $ Random\,($\,2\,\ldots\,N_{\,\max}\,$)
      \State $n\ \gets\ $ Random\,($\,2\,\ldots\,N\,$)
      \State $p\,(x)\ =\ -1\ +\ x\ +\ x^{\,n}$ \Comment{\it We initialize with this trinomial}
      \State $m\ \gets\ 2\,n\ -\ 1$ \Comment{$m_{\,1}\ -\ n\ \geqslant\ n\ -\ 1$}
      \While{$m\ \leqslant\ N$}
        \State $\Delta\,m\ \gets\ $ Random\,($\,0\,\ldots\,N\,-\,m\,$)
        \State $p\,(x)\ \gets\ p\,(x)\ +\ x^{\,m\,+\,\Delta\,m}$
        \State $m\ \gets\ m\ +\ \Delta\,m\ +\ n\ -\ 1$ \Comment{$m_{\,s\,+\,1}\ -\ m_{\,s}\ \geqslant\ n\ -\ 1$}
      \EndWhile
      \If{IrreducibilityTest\,$\bigl(p\,(x)\bigr)\ \equiv$ \textbf{True}}
        \State Irreducible\,[$k$] $\ \gets\ 1$
      \EndIf
    \EndFor
    \State $\mathds{P}\ \approx\ \frac{1}{M}\;\sum_{\,k\,=\,1}^{\,M}\;$Irreducible\,[$k$] \Comment{Approximate probability by frequency}
  \end{algorithmic}
  \bigskip
  \caption{\small\em Pseudo-code of the PARI/GP program used to estimate the probability to find an irreducible polynomial in the class $\B_{\,N_{\,\max}}\ =\ \cup_{\,k\,=\,1}^{\,N_{\,\max}}\;\B_{\,k}\,$.}
  \label{algo:mc}
\end{algorithm}

The following PARI/GP script estimates the probability of finding a sparse irreducible polynomial with coefficients in $\{-1,\, 0,\, 1\}$ in the class $\B\,$:
\begin{lstlisting}[label=list]
/* First, we increase the stack size: */
default(parisize, 1073741824); /* 1 Gb */

/* Search horizon in polynomial degree: */
Nmax = 3000;

/* Number of Monte-Carlo runs */
M = 4000;

/* The vector, where we stock the results of the irreducibility test: */
Irred = vector(M);

printf("Some information about computation:\n");
printf("  ->  Maximal polynomial degree: %d\n", Nmax);
printf("  -> Number of Monte-Carlo runs: %d\n\n", M);
printf("Computations started. Please, wait...");

ts = getabstime(); /* Record start time */
/* The main loop over realizations!: */
for (i = 1, M, {
  printf("iter = %d\n", i);
  N = 2 + random(Nmax - 1);
  n = 2 + random(N - 1);
  /* print(n); */
  /* P is the vector of coefficients */
  P = concat(concat(1, vector(n-2)), [1, -1]);
  p = length(P); /* we shall need it below */
  m = 2*n - 1; /* the next term has the degree >= m */
  while (m <= N,
    s = m + random(N - m + 1);
    P = concat(concat(1, vector(s - p)), P);
    p = length(P);
    m = s + n - 1;
  );
  pp = Pol(P, x); /* Convert vector to the polynomial: */
  /* print(pp); */
  if (polisirreducible(pp), /* if polynomial is irreducible, we note it */
    Irred[i] = 1;
  );
});
te = getabstime(); /* Simulation end time */
printf("Done. Execution time = %.3f s.\n", (te-ts)/1000);

/* Let's do some statistical analysis of the obtained data */
Mean = vecsum(Irred)/M;
Var  = 0.0;
for (i = 1, M, {
  Var += (Irred[i] - Mean)^2;
});
Var = sqrt(Var/(M - 1));
Err = 1.645*Var/sqrt(M);

printf("Estimated probability: %1.3f\n", Mean);
printf("Estimated 90%%-error  : %1.5f\n", Err);
printf("Confidence interval  : [%1.5f, %1.5f]\n",\
 max(0.0, Mean - Err), min(Mean + Err, 1.0));
\end{lstlisting}


\bigskip\bigskip
\addcontentsline{toc}{section}{References}
\bibliographystyle{abbrv}
\bibliography{biblio}

\begin{thebibliography}{10}

\bibitem{Borst2017}
C.~Borst, E.~Boyd, C.~Brekken, S.~Solberg, M.~M. Wood, and P.~M. Wood.
\newblock {Irreducibility of Random Polynomials}.
\newblock {\em Experimental Mathematics}, pages 1--9, may 2017.

\bibitem{Dobrowolski1979}
E.~Dobrowolski.
\newblock {On a question of Lehmer and the number of irreducible factors of a
  polynomial}.
\newblock {\em Acta Arith.}, 34(4):391--401, 1979.

\bibitem{Filaseta1999}
M.~Filaseta.
\newblock {On the Factorization of Polynomials with Small Euclidean Norm}.
\newblock In K.~Gyory, H.~Iwaniec, and J.~Urbanowicz, editors, {\em Number
  Theory in Progress}, pages 143--163, Zakopane, Poland, 1999. Walter De
  Gruyter.

\bibitem{Filaseta2006}
M.~Filaseta, C.~Finch, and C.~Nicol.
\newblock {On three questions concerning 0, 1-polynomials}.
\newblock {\em Journal de th{\'{e}}orie des nombres de Bordeaux},
  18(2):357--370, 2006.

\bibitem{Finch2006}
C.~Finch and L.~Jones.
\newblock {On the Irreducibility of {\{}-1,0,1{\}} - Quadrinomials}.
\newblock {\em Integers}, 6:{\#}A16, 2006.

\bibitem{Flatto1994}
L.~Flatto, J.~C. Lagarias, and B.~Poonen.
\newblock {The zeta function of the beta transformation}.
\newblock {\em Ergodic Theory and Dynamical Systems}, 14(2):237--266, jun 1994.

\bibitem{Konyagin1999}
S.~V. Konyagin.
\newblock {On the number of irreducible polynomials with 0, 1 coefficients}.
\newblock {\em Acta Arith.}, 88(4):333--350, 1999.

\bibitem{Ljunggren1960}
W.~Ljunggren.
\newblock {On the Irreducibility of Certain Trinomials and Quadrinomials.}
\newblock {\em Mathematica Scandinavica}, 8:65, dec 1960.

\bibitem{Lothaire2002}
M.~Lothaire.
\newblock {\em {Algebraic Combinatorics on Words}}.
\newblock Cambridge University Press, Cambridge, 2002.

\bibitem{Mignotte1999}
M.~Mignotte and D.~Stefanescu.
\newblock {On the roots of lacunary polynomials}.
\newblock {\em Mathematical Inequalities {\&} Applications}, 2(1):1--13, 1999.

\bibitem{Mills1985}
W.~H. Mills.
\newblock {The Factorization of Certain Quadrinomials}.
\newblock {\em Mathematica Scandinavica}, 57(1):44--50, 1985.

\bibitem{Odlyzko1993}
A.~M. Odlyzko and B.~Poonen.
\newblock {Zeros of Polynomials with 0, 1 Coefficients}.
\newblock {\em Enseign. Math.}, 39:317--348, 1993.

\bibitem{Schinzel1969}
A.~Schinzel.
\newblock {Reducibility of Lacunary Polynomials I}.
\newblock {\em Acta Arith.}, 16:123--160, 1969.

\bibitem{Schinzel1976}
A.~Schinzel.
\newblock {On the Number of Irreducible factors of a polynomial}.
\newblock {\em Colloq. Math. Soc. Janos Bolyai}, 13:305--314, 1976.

\bibitem{Schinzel1978}
A.~Schinzel.
\newblock {Reducibility of lacunary polynomials III}.
\newblock {\em Acta Arith.}, 34(3):227--266, 1978.

\bibitem{Schinzel1983}
A.~Schinzel.
\newblock {On the number of irreducible factors of a polynomial II}.
\newblock {\em Ann. Polon. Math.}, 42(1):309--320, 1983.

\bibitem{Selmer1956}
E.~S. Selmer.
\newblock {On the irreducibility of certain trinomials}.
\newblock {\em Mathematica Scandinavica}, 4(2):287--302, dec 1956.

\bibitem{Solomyak1994}
B.~Solomyak.
\newblock {Conjugates of Beta-Numbers and the Zero-Free Domain for a Class of
  Analytic Functions}.
\newblock {\em Proc. Lond. Math. Soc.}, s3-68(3):477--498, may 1994.

\bibitem{Verger-Gaugry2008}
J.-L. Verger-Gaugry.
\newblock {Uniform Distribution of the Galois Conjugates and Beta-Conjugates of
  a Parry Number Near the Unit Circle and Dichotomy of Perron Numbers}.
\newblock {\em Uniform distribution theory}, 3:157--190, 2008.

\bibitem{Verger-Gaugry2016}
J.-L. Verger-Gaugry.
\newblock {On the Conjecture of Lehmer, Limit Mahler Measure of Trinomials and
  Asymptotic Expansions}.
\newblock {\em Uniform distribution theory}, 11(1):79--139, jan 2016.

\bibitem{Verger-Gaugry2018}
J.-L. Verger-Gaugry.
\newblock {A Proof of the Conjecture of Lehmer and of the Conjecture of
  Schinzel-Zassenhaus}.
\newblock {\em Submitted}, pages 1--164, 2018.

\end{thebibliography}
\bigskip\bigskip

\end{document}